\documentclass[10pt]{article}   
\usepackage{graphicx}
\usepackage{latexsym,amsfonts,amsmath,theorem,amssymb}
\usepackage{tikz}
\usepackage{stmaryrd,array,tabularx,bbm}
\usepackage{amsmath,booktabs,ctable,threeparttable,boxedminipage}
\usepackage{pstricks,graphicx,epstopdf}
\usepackage{a4wide}
\usepackage{color}
\usepackage{caption}
\usepackage{theorem}
\usepackage{hyperref}
\usepackage{cleveref}
\usepackage{framed}
\usepackage{blindtext}
\usepackage{float}
\usepackage{todonotes}
\usepackage{lineno}
\usepackage{multirow}
\usepackage[caption=false]{subfig}
\usepackage{algorithm,algpseudocode}
\usepackage{blkarray}
\usepackage{xcolor}
\usepackage{geometry}
\usepackage{blkarray}
\usepackage{arydshln}

\algrenewcommand\algorithmicrequire{\textbf{Input:}}
\algrenewcommand\algorithmicensure{\textbf{Output:}}

\newtheorem{rem}{Remark}[section]
\numberwithin{equation}{section}
\numberwithin{figure}{section}
\numberwithin{table}{section}

\newtheorem{theorem}{Theorem}[section]

\newtheorem{proposition}{Proposition}[section]

\newtheorem{definition}{Definition}[section]

\newenvironment{proof}{\begin{trivlist}
    \item[\hskip\labelsep{\bf Proof.}]}{$\hfill\Box$\end{trivlist}}

{\theoremstyle{plain} \theorembodyfont{\rmfamily}
}

\numberwithin{equation}{section}
\numberwithin{figure}{section}
\numberwithin{table}{section}

\newcommand{\vertiii}[1]{{\left\vert\kern-0.25ex\left\vert\kern-0.25ex\left\vert #1 
\right\vert\kern-0.25ex\right\vert\kern-0.25ex\right\vert}}
\title{Generalizing the SVD of a matrix under non-standard inner product and its applications to linear ill-posed problems}
\author{Haibo Li
\thanks{School of Mathematics and Statistics, The University of Melbourne, Parkville, VIC 3010, Australia.
\href{mailto:haibo.li@unimelb.edu.au}{haibo.li@unimelb.edu.au}}
}
\date{}

\begin{document}
\maketitle

\begin{abstract}
The singular value decomposition (SVD) of a matrix is a powerful tool for many matrix computation problems. In this paper, we consider generalizing the standard SVD to analyze and compute the regularized solution of linear ill-posed problems that arise from discretizing the first kind Fredholm integral equations. For the commonly used quadrature method for discretization, a regularizer of the form $\|x\|_{M}^2:=x^TMx$ should be exploited, where $M$ is symmetric positive definite. To handle this regularizer, we give the weighted SVD (WSVD) of a matrix under the $M$-inner product. Several important applications of WSVD, such as low-rank approximation and solving the least squares problems with minimum $\|\cdot\|_M$-norm, are studied. We propose the weighted Golub-Kahan bidiagonalization (WGKB) to compute several dominant WSVD components and a corresponding weighted LSQR algorithm to iteratively solve the least squares problem. All the above tools and methods are used to analyze and solve linear ill-posed problems with the regularizer $\|x\|_{M}^2$. A WGKB-based subspace projection regularization method is proposed to efficiently compute a good regularized solution, which can incorporate the prior information about $x$ encoded by the regularizer $\|x\|_{M}^2$. Several numerical experiments are performed to illustrate the fruitfulness of our methods.
\end{abstract}


\tableofcontents

\section{Introduction}\label{sec1}
The singular value decomposition (SVD) is a well-known matrix factorization tool for diagonalizing a matrix with arbitrary shape \cite{doi:10.1137/1035134}. It generalizes the eigen-decomposition of a square normal matrix to any $m\times n$ matrix. Let $A\in\mathbb{R}^{m\times n}$, then there exist orthogonal matrices $\hat{U}=(\hat{u}_1,\dots,\hat{u}_m)\in\mathbb{R}^{m\times m}$ and $\hat{V}=(\hat{v}_1,\dots,\hat{v}_n)\in\mathbb{R}^{n\times n}$, such that
\begin{equation}\label{svd1}
  \hat{U}^{\top}A\hat{V} = \hat{\Sigma}
\end{equation} 
with
\begin{equation}
  \hat{\Sigma} = \begin{pmatrix}
    \hat{\Sigma}_{q} \\ \mathbf{0}
  \end{pmatrix},  \  m\geq n \ \ \ \
  \mathrm{or} \ \ \ \
  \hat{\Sigma} = \begin{pmatrix}
    \hat{\Sigma}_{q} & \mathbf{0}
  \end{pmatrix},  \  m<n,
\end{equation}
where $q=\min\{m,n\}$ and $\hat{\Sigma}_{q}=\mathrm{diag}(\hat{\sigma}_1,\dots,\hat{\sigma}_q)\in\mathbb{R}^{q\times q}$ is a diagonal matrix. The real values $\hat{\sigma}_1\geq\dots\geq\hat{\sigma}_q\geq 0$ are called singular values, and the corresponding vectors $u_i$ and $v_i$ are called right and left singular vectors, respectively. 

The applications of SVD have many aspects. The mathematical applications include determining the rank, range and null spaces of a matrix, computing the pseudoinverse of a matrix, determining a low-rank approximation of a matrix, solving linear least squares problems, solving discrete linear ill-posed problems, and many others \cite{eckart1936approximation, Hansen1998, Bjorck1996, Golub2013}. We will review a part of them in Section \ref{sec2}. Besides, SVD is also extremely useful in all areas of science and engineering, such as signal processing, image processing, principal component analysis, control theory, recommender systems and many others \cite{smith1981multivariable,resnick1997recommender, Hansen2006, jolliffe2016principal}. For a large-scale matrix $A$, a partial SVD can be computed by using the Golub-Kahan bidiagonalization (GKB), which generates two Krylov subspaces and projects $A$ onto these two subspaces to get a small-scale bidiagonal matrix \cite{golub1965}. The bidiagonal structure of the projected matrix makes it convenient to develop efficient algorithms. For example, the dominant singular values and corresponding vectors of $A$ can be well approximated by the SVD of the projected bidiagonal matrix \cite{demmel1990accurate}. It is also shown in \cite{simon2000low} that a good low-rank approximation of $A$ can be directly obtained from the GKB of $A$ without directly computing any SVD. For large sparse least squares problem of the form $\min_{x\in\mathbb{R}^{n}}\|Ax-b\|_2$, the most commonly used LSQR solver is also based on GKB \cite{Paige1982}.

In this paper, we focus on generalizing the SVD to analyze and compute the regularized solution of the first kind Fredholm equation \cite{kress2013linear}
\begin{equation}\label{fred0}
	g(s) = \int_{t_1}^{t_2}K(s,t)f(t)\mathrm{d}t + \sigma \dot W(s)
\end{equation}
with $t\in[t_1, t_2]$ and $s\in [s_1, s_2]$, where $K(s,t)\in L^{2}\left([t_1,t_2]\times[s_1,s_2]\right)$ is a square-integrable kernel function, $\dot W(s)$ is the Gaussian white noisy, more precisely, the generalized derivative of the standard Brown motion $W(s)$ with $s\in [s_1, s_2]$ \cite{karatzas2012brownian}, $\sigma$ is the scale of the noise, and $g(s)$ is the observation. The aim is to recover the unknown $f(t)$ from the noisy observation. To solve this problem, the first step is to discretize the above integral equation. Two commonly used discretization methods are the quadrature method and Galerkin method \cite{kirsch2011introduction}, and we focus on the first one in this paper. In the quadrature method, a quadrature rule with grid points $\{p_1,\dots,p_n\}$ and corresponding weights $\{w_1,\dots,w_n\}$ are chosen to approximate the integral as
\begin{equation}\label{Fred_disc}
  \int_{t_1}^{t_2}K(s,t)f(t)\mathrm{d}t \approx
  \sum_{i=1}^{n}w_iK(s,p_i)
\end{equation}
and the underlying unknown function we want to recover becomes the $n$-dimensional vector $x_{\text{true}}=(f(p_1),\dots,f(p_n))^{\top}$. The observations are chosen from $m$ grid points in $[s_1,s_2]$ to get a $m$-dimensional noisy vector $b$. After discretization, the above integral equation \eqref{fred0} becomes
\begin{equation}\label{ill-posed1}
  b = Ax + e ,
\end{equation}
where $e$ is a discrete Gaussian noise vector.

The above discrete linear system is usually ill-posed in the sense that $A$ is extremely ill-conditioned with singular values decreasing toward zero without a noticeable gap \cite{Hansen1998}. As a result, the naive solution to the least squares problem $\min_{x\in\mathbb{R}^{n}}\|Ax-b\|_2$ will deviate very far from the true solution. Tikhonov regularization is usually used to handle the ill-posedness property of this problem \cite{Tikhonov1977}. For a Gaussian white noise $e$, the standard-form Tikhonov regularization has the form 
\begin{equation}\label{tikh0}
  \min_{x\in\mathbb{R}^n}\|Ax-b\|_{2}^2+\lambda\|x\|_{2}^2 .
\end{equation}
However, we emphasis that the regularization term $\|x\|_{2}^2$ arises from the discretized version of $\|f\|_{L^{2}([t_1,t_2])}$; here $L^{2}([t_1,t_2])$ is the space of square-integrable functions define on $[t_1,t_2]$ with Lebesgue measure. If all the weights in the quadrature rule for approximating the integral in \eqref{Fred_disc} have the same value, then the vector norm $\|x\|_2$ is a good approximation to $\|f\|_{L^{2}([t_1,t_2])}$. If the weights have different values, a good approximation to $\|f\|_{L^{2}([t_1,t_2])}$ is $\|x\|_{M}$ instead of $\|x\|_2$, where $\|x\|_M=(x^{\top}Mx)^{1/2}$ with $M=\mathrm{diag}(w_1,\dots,w_n)$. In this case, the standard-form Tikhonov regularization should be replaced by
\begin{equation}\label{tikh1}
  \min_{x\in\mathbb{R}^n}\|Ax-b\|_{2}^2+\lambda\|x\|_{M}^2 .
\end{equation}

Using the standard SVD of $A$, we can analyze and compute the ill-posed problem \eqref{ill-posed1} with the regularizer $\|x\|_{2}^2$. For example, the naive solution to $\min_{x\in\mathbb{R}^{n}}\|Ax-b\|_2$ can be expressed as the SVD expansion form $x_{LS} = A^{\dag}b=\sum_{i=1}^{r}\frac{\hat{u}_{i}^{\top}b}{\hat{\sigma}_i}\hat{v}_i$, where $r$ is the rank of $A$, motivating the so-called ``truncated SVD'' (TSVD) regularization method, which truncates the first $k$ dominant SVD expansion from $x_{LS}$ to discard those highly amplified noisy components \cite{Hansen1987}. Moreover, the Tikhonov regularized solution to \eqref{tikh0} can be expressed as the filtered SVD expansion form $x_{\lambda} = \sum_{i=1}^{r}\frac{\hat{\sigma}_{i}^2}{\hat{\sigma}_{i}^2+\lambda}\frac{\hat{u}_{i}^{\top}b}{\hat{\sigma}_i}\hat{v}_i$, which tells us that the regularization parameter $\lambda$ should be chosen such that the filter factors $f_i:=\frac{\hat{\sigma}_{i}^2}{\hat{\sigma}_{i}^2+\lambda}\approx 1$ for small index $i$ and $f_i\approx 0$ for large $i$ to suppress noisy components. Correspondingly, the GKB process for computing SVD can also be exploited to design iterative regularization algorithms \cite{Oleary1981, Bjorck1988}. The LSQR solver with early stopping rules is a standard algorithm to handle the regularizer $\|x\|_{2}^2$. Specifically, the regularization property of LSQR with early stopping rules has been analyzed using the SVD of $A$, where the $k$-th LSQR solution also has a filtered SVD expansion form \cite{Hansen1998}.

To analyze the ill-posed problem \eqref{ill-posed1} with the regularizer $\|x\|_{M}^2$, in this paper, we generalize the standard SVD of $A$ to a new form under a non-standard inner product, which means that the $2$-inner product in $\mathbb{R}^n$ is replaced by the $M$-inner product $\langle z, z'\rangle:=z^{\top}Mz'$. The core idea is to treat the matrix $A\in\mathbb{R}^{m\times n}$ as a linear operator between the two finite-dimensional Hilbert spaces $\left(\mathbb{R}^n, \langle \cdot, \cdot \rangle_M \right)$ and $\left(\mathbb{R}^m, \langle \cdot, \cdot \rangle_2 \right)$. The attempt to generalize SVD is not new, which can go back to Van Loan in 1976 \cite{Van1976}, where he proposed the generalized SVD (GSVD) for a matrix pair. The GSVD is a powerful tool to analyze the general-form Tikhonov regularization term $\|Lx\|_{2}^2$, where $L\in\mathbb{R}^{p\times n}$ with $p\leq n$ is a regularization matrix. In recent years, there have been several other generalized forms of SVD, such as the weighted SVD of different forms proposed in \cite{sergienko2015weighted,kyrchei2017weighted,jozi2018weighted}. Specifically, in \cite{jozi2018weighted}, the authors used the weighted SVD to solve the discrete ill-posed problem arised from \eqref{fred0} when $f(t)$ and $g(s)$ are discretized on the same grid points with $[t_1,t_2]=[s_1,s_2]$ and $m=n$. In their paper, both the data fidelity term and regularization term use the weighted $M$-norm. However, we emphasize that for the discrete observation vector $b$ with Gaussian white noise, the most appropriate form for the data fidelity term should be $\|Ax-b\|_{2}^2$. In this paper, we generalize a new form of SVD, investigate its properties and propose numerical algorithms for its computation. Then we use this generalized SVD to instigate the solution of linear least squares problems with minimum $\|x\|_M$ norm and propose an iterative algorithm for this problem. This iterative algorithm with proper early stopping rules can be used to handle the discrete ill-posed problem with regularization term $\|x\|_{M}^2$.
 
The main contributions of this paper are listed as follows.
\begin{itemize}
  \item For any symmetric positive definite matrix $M\in\mathbb{R}^{n\times n}$, we generalize a new form SVD of $A\in\mathbb{R}^{m\times n}$ where the right singular vectors constitute an $M$-orthogonal basis of $\left(\mathbb{R}^n, \langle \cdot, \cdot \rangle_M \right)$. This new generalized SVD is also named the weighted SVD (WSVD), which shares many similar good properties as the standard SVD.
  \item We study some basic properties of WSVD and several applications, including its relations with the null space and range space of $A$, a new form of low-rank approximation of $A$ based on WSVD, and the WSVD form expression of the solution to $\min_{x\in\mathbb{R}^{n}}\|Ax-b\|_2$ with minimum $\|x\|_M$ norm.
  \item We propose a weighted Golub-Kahan bidiagonalization (WGKB) process, which can be used to compute several dominant WSVD components. A WGKB-based iterative algorithm for solving the least squares problems is also proposed. It is a weighted form corresponding to the standard LSQR algorithms, thereby we name it weighted LSQR (WLSQR).
  \item Using WSVD, we analyze the solution of the Tikhonov regularization problem \eqref{tikh1}. In order to utilize the information from the regularizer $\|x\|_{M}^2$ and avoid selecting $\lambda$ in advance, we propose the subspace projection regularization (SPR) method. We show that the WGKB-based SPR method is just the WLSQR with early stopping rules, which can efficiently compute a satisfied regularized solution.
\end{itemize}
In the numerical experiments part, we use four test examples of the first kind Fredholm equations to test our algorithm. With Simpson's rule for discretizing the integral, we show that the proposed algorithm has better performance than the standard LSQR solver for regularizing the original problem.

The paper is organized as follows. In Section \ref{sec2} we review some basic properties and applications of SVD. In Section \ref{sec3}, we propose the weighted SVD of $A$ with weight matrix $M$ and investigate its properties. In Section \ref{sec4}, we propose the WGKB process to compute several dominant WSVD components and the WLSQR algorithms for solving least squares problems. All the above methods are used in Section \ref{sec5} to analyze and develop iterative regularization methods for linear ill-posed problems with regularizer $\|x\|_{M}^2$. In Section \ref{sec6}, we use several numerical examples to illustrate the effectiveness of the new methods. Finally, we conclude the paper in Section \ref{sec7}.

Throughout the paper, denote by $\mathcal{R}(C)$ and $\mathcal{N}(C)$ the range space and null space of a matrix $C$, by $I_{k}$ the $k$-by-$k$ identity matrix. We denote by $\mathbf{0}$ a zero vector or matrix with its order clear from the context.

\section{Properties of SVD and its applications}\label{sec2}
In this section, we review several important properties of the SVD and its applications to linear ill-posed problems. They motivate us to generalize the SVD to the non-standard inner product case, which can be applied to handle the regularizer of the form \eqref{tikh1}.

Let us start from the connections between the SVD, Schatten $p$-norm \cite[\S 7.4.7]{horn2012matrix}, and low-rank approximation of a matrix \cite[\S 2.4]{Golub2013}.

\begin{theorem}[Schatten $p$-norm] \label{pnorm}
  Suppose the SVD of a matrix $A$ is as in \eqref{svd1}. For any integer $1\leq p\leq \infty$, define
  \begin{equation*}
    \|A\|_p = \left(\sum_{i=1}^{\min\{m,n\}}\hat{\sigma}_{i}^p\right)^{1/p}.
  \end{equation*}
  Then $\|\cdot\|_p$ is a matrix norm on $\mathbb{R}^{m\times n}$, called the Schatten $p$-norm.
\end{theorem}

The specific choice of $p$ yields several commonly used matrix norms:
\begin{enumerate}
  \item $p=1$: gives the nuclear norm. It is commonly used in low-rank matrix completion algorithms.
  \item $p=2$: gives the Frobenius norm (often denoted by $\|\cdot\|_F$).
  \item $p=\infty$: gives the spectral norm (often denoted by $\|\cdot\|_2$).
\end{enumerate}

All the Schatten norms are unitarily invariant, which means that $\|A\|_p=\|\bar{U}^{\top}A\bar{V}\|_p=\|A\|_{p}$ for any matrix $A$ and all unitary matrices $\bar{U}$ and $\bar{V}$. In this paper, we focus on the spectral norm and use the popular notation $\|\cdot\|_2$ to denote it. From another definition of the spectral norm, we also have $\|A\|_2=\max_{x\neq\mathbf{0}}\frac{\|Ax\|_2}{\|x\|_2}=\hat{\sigma}_1$.

One of the reasons that SVD is so widely used is that it can be used to find the best low-rank approximation to a matrix. The following low-rank approximation property is often used such as in data compression, image compression and recommender systems.

\begin{theorem}[Eckhart-Young-Mirsky]
  Suppose $k<\mathrm{rank}(A)=r$ and let 
  \begin{equation}
    \hat{A}_k = \hat{U}_{k}\hat{\Sigma}_k\hat{V}_{k}^{\top}=\sum_{i=1}^{k}\hat{\sigma}_i\hat{u}_i\hat{v}_{i}^{\top} ,
  \end{equation}
  where $\hat{U}_k$ and $\hat{V}_k$ contain the first $k$ columns of $\hat{U}$ and $\hat{V}$, and $\hat{\Sigma}_k$ is the first $k$-by-$k$ part of $\hat{\Sigma}$.
  Then 
  \begin{equation}
    \min_{\mathrm{rank}(X)\leq k}\|A-X\|_2=\|A-\hat{A}_k\|_2 = \hat{\sigma}_{k+1}.
  \end{equation}
\end{theorem}

The Moore-Penrose pseudoinverse of a matrix is the most widely known generalization of the inverse matrix \cite[5.5.2]{Golub2013}. Using the SVD of $A$, we can give its explicit form: 
\begin{equation}
  A^{\dag}= \hat{V}\hat{\Sigma}^{\dag}\hat{U}^{\top}, \ \ \  
\hat{\Sigma}^{\dag} = \begin{pmatrix}
  \hat{\Sigma}_{r}^{-1} & \mathbf{0} \\
  \mathbf{0} & \mathbf{0}
\end{pmatrix} .
\end{equation}
One application of the Moore-Penrose pseudoinverse is to solve the rank-deficient least squares problems \cite[\S 5.5.1]{Golub2013}.

\begin{theorem}
  Let $\mathrm{rank}(A)=r$ (can be smaller than $\min\{m,n\}$). Then the rank-deficient least squares problem with a minimum 2-norm 
  \begin{equation}
    \min_{x\in \mathcal{S}}\|x\|_2, \ \ \ \mathcal{S}=\{x: \|Ax-b\|_2=\mathrm{min}\}
  \end{equation}
  has a unique solution
  \begin{equation}\label{LS}
    x_{LS} = A^{\dag}b
    =\sum_{i=1}^{r}\frac{\hat{u}_{i}^{\top}b}{\hat{\sigma}_i}\hat{v}_i .
  \end{equation}
\end{theorem}

For linear ill-posed problem with Tikhonov regularization \eqref{tikh0}, the regularized solution has a similar expression to \eqref{LS} but with additional fiter factors \cite[\S 4.2]{Hansen1998}:
\begin{equation}\label{fil_sol}
  x_{\lambda} = \sum_{i=1}^{r}\frac{\hat{\sigma}_{i}^2}{\hat{\sigma}_{i}^2+\lambda}\frac{\hat{u}_{i}^{\top}b}{\hat{\sigma}_i}\hat{v}_i,
\end{equation}
where the regularzation parameter $\lambda$ should be chosen such that the filter factors $f_{i}:=\frac{\hat{\sigma}_{i}^2}{\hat{\sigma}_{i}^2+\lambda}\approx 1$ for those small index $i$ corresponding to dominant information about $x_{\text{true}}$ and $f_i\approx 0$ for those large index $i$ to filter out noisy components.

For small-scale matrices, the SVD can be efficiently computed by a variant of QR algorithm \cite{golub1965} or Jocobi rotation procedure \cite{drmavc2008new1,drmavc2008new2}. For large-scale matrices, one commonly used SVD algorithm is based on the Golub-Kahan bidiagonalization (GKB), which applies a Lanczos-type iterative procedure to $A$ to generate two Krylov subspaces and project $A$ to be a small-scale bidiagonal matrix. Then the SVD of the reduced bidiagonal matrix is computed to approximate some dominant SVD components of $A$ \cite{fernando1994accurate}. The GKB process is also a standard procedure used in LSQR for iteratively solving large-scale least square problems. At the $k$-th step, it solves the following problem:
\begin{equation*}
  x_k = \min_{x\in\mathcal{S}_k}\|Ax-b\|_2, \ \ \ \mathcal{S}_k=\mathcal{K}_k(A^{\top}A,A^{\top}b):=\mathrm{span}\{(A^{\top}A)^i,A^{\top}b\}_{i=0}^{k-1} .
\end{equation*}
For large-scale linear ill-posed problems, the above approach is very efficient and fruitful for handling the $\|x\|_{2}^{2}$ regularization term, where an early stopping rule should be used to avoid containing too much noise in the iterative solution.

If the first kind Fredholm equation is discretized using weights $\{w_{1},\dots,w_n\}$ with different values in \eqref{Fred_disc}, then the corresponding approximation to $\|f\|_{L^2}$ is $\|x\|_M$ instead of $\|x\|_2$, where $M=\mathrm{diag}(w_{1},\dots,w_n)$. In this case, we need to consider the Tikhonov regularization problem \eqref{tikh0} and the corresponding iterative regularization methods. To analyze and solve this problem, in the following part, we generalize the SVD to the $M$-inner product case and propose corresponding iterative algorithms.

\section{Weighted SVD with non-standard inner-product}\label{sec3}
In this section, let $M\in\mathbb{R}^{n\times n}$ be a symmetric positive definite matrix. It can be either diagonal or non-diagonal. This matrix can introduce a new inner product in $\mathbb{R}^{n}$.

\begin{definition}
  For any $A\in\mathbb{R}^{m\times n}$, define the linear operator $\mathcal{A}: (\mathbb{R}^{n}, \langle\cdot,\cdot\rangle_{M}) \to (\mathbb{R}^{m}, \langle\cdot,\cdot\rangle_{2})$ as $\mathcal{A}: x\mapsto Ax$ for $x\in\mathbb{R}^{n}$ under the canonical bases of $\mathbb{R}^{n}$ and $\mathbb{R}^{m}$, where $\langle\cdot,\cdot\rangle_{2}$ is the standard 2-inner product and $\langle x,x'\rangle_{M}:=x^{\top}Mx'$ is called the $M$-inner product.
\end{definition}

The operator $\mathcal{A}$ is bounded since $(\mathbb{R}^{n}, \langle\cdot,\cdot\rangle_{M})$ and $(\mathbb{R}^{m}, \langle\cdot,\cdot\rangle_{2})$ are both finite dimensional Hilbert spaces. Thus, we can definite the operator norm of $\mathcal{A}$.

\begin{definition}
  Define the $M$-weighted norm of $A$ as
  \begin{equation}
    \|A\|_{M,2} := \|\mathcal{A}\|:= \max_{x\neq\mathbf{0}}\frac{\|Ax\|_2}{\|x\|_{M}}.
  \end{equation}
\end{definition}

Similar to the unitarily invariant property of $\|A\|_2$, we have the following property for $\|A\|_{M,2}$.

\begin{proposition}\label{lem1}
  Let $\tilde{U}\in\mathbb{R}^{m\times m}$ and $\tilde{V}\in\mathbb{R}^{n\times n}$ are 2- and $M$- orthogonal matrices, respectively, i.e. $\tilde{U}^{\top}\tilde{U}=I_{m}$ and $\tilde{V}^{\top}M\tilde{V}=I_{n}$. Then we have
  \begin{equation}
    \|\tilde{U}^{\top}A\tilde{V}\|_{2} = \|A\|_{M,2}
  \end{equation}
\end{proposition}
\begin{proof}
  Since $\tilde{V}^{\top}M\tilde{V}=I_{n}$ and $M$ is positive definite, it follows that $\tilde{V}$ is invertible. Thus, we have
  \begin{align*}
    \|\tilde{U}^{\top}A\tilde{V}\|_{2}
    &= \max_{x\neq\mathbf{0}}\frac{\|\tilde{U}^{\top}A\tilde{V}x\|_2}{\|x\|_2}
    = \max_{x\neq\mathbf{0}}\frac{\|\tilde{U}^{\top}A\tilde{V}x\|_2}{\|\tilde{V}x\|_M} \\
    &= \max_{y\neq\mathbf{0}}\frac{\|\tilde{U}^{\top}Ay\|_2}{\|y\|_M}
    = \max_{y\neq\mathbf{0}}\frac{\|Ay\|_2}{\|y\|_M}  \ \ \ \ \ (\mathrm{let} \ y=\tilde{V}x) \\
    &= \|A\|_{M,2},
  \end{align*} 
  where we have used $\|\tilde{V}x\|_{M}^2=x^{\top}\tilde{V}^{\top}M\tilde{V}x=x^{\top}x=\|x\|_{2}^2$. The proof is completed.
\end{proof}

The following result generates the SVD of a matrix $A$, which has a very similar form to \eqref{svd1}.

\begin{theorem}\label{thm:wsvd}
  Let $A\in\mathbb{R}^{m\times n}$, and $M\in\mathbb{R}^{n\times n}$ is symmetric positive definite. There exist 2-orthogonal matrix $U\in\mathbb{R}^{m\times m}$ and $M$-orthogonal matrix $V\in\mathbb{R}^{n\times n}$, and diagonal matrix $\Sigma_{r}=\mathrm{diag}(\sigma_1,\dots,\sigma_r)$ with $\sigma_1\geq \dots \geq \sigma_r>0$, such that 
  \begin{equation}\label{wsvd}
    U^{\top}AV = \Sigma := \begin{pmatrix}
      \Sigma_r & \mathbf{0} \\
      \mathbf{0} & \mathbf{0}
    \end{pmatrix}, \ \ \ 
    \Sigma \in \mathbb{R}^{m\times n} .
  \end{equation}
\end{theorem}
\begin{proof}
  By the definition of $\|A\|_{M,2}$, there exist vectors $v_1\in\mathbb{R}^{n}$ and $u_1\in\mathbb{R}^{m}$ such that $\|v_{1}\|_{M}=\|u_1\|_2=1$ and $Av_1=\sigma_1u_{1}$ with $\sigma_1=\|A\|_{M,2}$. Let $V_{2}\in\mathbb{R}^{n\times(n-1)}$ and $U_{2}\in\mathbb{R}^{m\times(m-1)}$ such that $V=(v_1, V_2)$ and $U=(u_1, U_2)$ are $M$- and 2-orthogonal, respectively. Then we get
  \begin{equation*}
    U^{\top}AV = 
    \begin{pmatrix}
      u_{1}^{\top}Av_1 & u_{1}^{\top}AV_{2} \\
      U_{2}^{\top}Av_{1} & U_{2}^{\top}AV_{2}
    \end{pmatrix} =:
    \begin{pmatrix}
      \sigma_{1} & x^{\top} \\
      \mathbf{0} & B
    \end{pmatrix} =: A_{1} ,
  \end{equation*}
  where $x\in\mathbb{R}^{n-1}$ and $B\in\mathbb{R}^{(m-1)\times(n-1)}$. By Lemma \ref{lem1} we have $\|A_{1}\|_{2}=\|A\|_{M,2}=\sigma_1$. Let $\tilde{x}=(\sigma_{1}, x^{\top})^{\top}$. It follows that
  \begin{align*}
    \sigma_{1}^2\geq \frac{\|A_{1}\tilde{x}\|_{2}^2}{\|\tilde{x}\|_{2}^2}
    = \left\|\begin{pmatrix}
      \sigma_{1}^2+x^{\top}x \\ Bx
    \end{pmatrix}
    \right\|_{2}^{2} \bigg/ \|\tilde{x}\|_{2}^2 \geq \|\tilde{x}\|_{2}^2=\sigma_{1}^2+x^{\top}x,
  \end{align*}
  which leads to $x=\mathbf{0}$. Therefore, we have $U^{\top}AV=\begin{pmatrix}
    \sigma_{1} & \mathbf{0}^{\top} \\
    \mathbf{0} & B
  \end{pmatrix}$. Now \eqref{wsvd} can be obtained by using mathematical induction. Since $U$ and $V$ are invertible, it follows that $\mathrm{rank}(A)=r=\mathrm{rank}(\Sigma_r)$. The proof is completed.
\end{proof}

The main difference between the two forms \eqref{svd1} and \eqref{wsvd} is that the right vectors $\{v_i\}$ are $M$-orthonormal. We call \eqref{wsvd} the \textit{weighted SVD} (WSVD) of $A$ with weight matrix $M$. For $M=I_n$, it is the same as the standard SVD. Similar to the standard SVD, the WSVD can be used to analyze and solve many problems with a non-standard 2-norm. Specifically, it can be used to analyze and develop efficient algorithms for linear ill-posed problems with the $\|x\|_{M}^2$ regularization term.

Note that $V^{\top}MV=I_n$ implies that $VV^{\top}MV=V$. Multiplying $V^{-1}$ from the right-hand side, we get $VV^{\top}=M^{-1}$. Therefore, from \eqref{wsvd} we get
\begin{equation}\label{wsvd1}
  A = U\begin{pmatrix}
    \Sigma_r & \mathbf{0} \\
    \mathbf{0} & \mathbf{0}
  \end{pmatrix}V^{\top}M,
\end{equation}
From \eqref{wsvd} and \eqref{wsvd1} we have
\begin{align}
  Av_i &= \sigma_{i}u_i, \\
  A^{\top}u_i &= \sigma_{i}Mv_{i}.
\end{align}
Also, we have the WSVD expansion form of $A$: $A=\sum_{i=1}^{r}\sigma_{i}u_{i}v_{i}^{\top}M$. Besides, the range space and null space of $A$can be explicitly written as
\begin{align*}
  \mathcal{R}(A) &= \mathrm{span}\{u_1,\dots,u_r\}, \\
  \mathcal{N}(A) &= \mathrm{span}\{v_{r+1},\dots,v_n\},
\end{align*}
where $\{v_{r+1},\dots,v_n\}$ is an $M$-orthonormal basis of $\mathcal{N}(A)$.

Using WSVD, the Eckhart-Young-Mirsky theorem for low-rank approximation of a matrix under the $\|\cdot\|_{M,2}$ norm has the following form.

\begin{theorem}
  For any integer $1\leq k<r$, we have
  \begin{equation}
    \min_{\mathrm{X}\leq k}\|A-X\|_{M,2} \geq \sigma_{k+1},
  \end{equation}
  where the minimum can be achieved if $X=A_{k}:=\sum_{i=1}^{k}\sigma_{i}u_{i}v_{i}^{\top}M$.
\end{theorem}
\begin{proof}
  First, if $X=A_{k}:=\sum_{i=1}^{k}\sigma_{i}u_{i}v_{i}^{\top}M$, we have $\mathrm{rank}(X)=k$, and by Proposition \ref{lem1} we have
  \begin{align*}
    \|A-X\|_{M,2} = \|U^{\top}(A-X)V\|_2 
    = \left\|U^{\top}U\begin{pmatrix}
      \mathbf{0} & & & & \\
        & \sigma_{k+1} & & & \\
        &  & \ddots & &  \\
        &  &  & \sigma_{r} & \\
        &  &  &    & \mathbf{0}
    \end{pmatrix}V^{\top}MV\right\|_{2} = \sigma_{k+1} .
  \end{align*} 
  Thus, we only need to prove $\|A-X\|_{M,2}\leq \sigma_{k+1}$ for any $X\in\mathbb{R}^{m\times n}$ with $\mathrm{rank}(X)=k$. Fot such $X$ we have $\mathrm{dim}(\mathcal{N}(X))=n-k$, thereby we can find $M$-orthonormal vectors $\{w_{1},\dots,w_{n-k}\}$ such that $\mathcal{N}(X)=\mathrm{span}\{w_{1},\dots,w_{n-k}\}$. Notice that $\mathcal{N}(X)\cap \mathrm{span}\{v_{1},\dots,v_{k+1}\}\neq\{\mathbf{0}\}$ since the sum of dimensions of these two subspaces is $n+1$. Let $z$ be a nonzero vector in the intersection of the above two subspaces and $\|z\|_{M}=1$. Using the WSVD of $A$, we get
  \begin{equation*}
    Az = \sum_{i=1}^{r}\sigma_{i}u_{i}(v_{i}^{\top}Mz)
       = \sum_{i=1}^{k+1}\sigma_{i}u_{i}(v_{i}^{\top}Mz),
  \end{equation*}
  since $z$ is $M$-orthogonal to $v_{k+2},\dots,v_{n}$. It follows that
  \begin{align*}
    \|A-X\|_{M,2}^2 \geq \frac{\|(A-X)z\|_{2}^2}{\|z\|_{M}^2} = \|Az\|_{2}^2
    = \sum_{i=1}^{k+1}\sigma_{i}^{2}(v_{i}^{\top}Mz)^2
    \geq \sigma_{k+1}^2\sum_{i=1}^{k+1}(v_{i}^{\top}Mz)^2 .
  \end{align*}
  Since $\|z\|_{M}^2=z^{\top}Mz=z^{\top}MVV^{\top}Mz=\|V^{\top}Mz\|_{2}^2$, where we used $VV^{\top}=M^{-1}$, we have 
  \begin{equation*}
    \sum_{i=1}^{k+1}(v_{i}^{\top}Mz)^2 = \sum_{i=1}^{n}(v_{i}^{\top}Mz)^2
    = \|V^{\top}Mz\|_{2}^2 = \|z\|_{M}^2=1.
  \end{equation*}
  We finnaly obtain $\|A-X\|_{M,2}\geq\sigma_{k+1}$.
\end{proof}

For the rank-deficient least squares problem, we can write the solution set by using the WSVD, which is convenient to find the unique minimum $\|\cdot\|_M$ solution.
\begin{theorem}\label{thm6}
  For the linear least squares problems $\min_{x\in\mathbb{R}^n}\|Ax-b\|_2$, the set of all solutions is
  \begin{equation}\label{sol_set}
    \mathcal{X} = \sum_{i=1}^{r}\frac{u_{i}^{\top}b}{\sigma_{i}}v_{i} + \mathrm{span}\{v_{r+1},\dots,v_n\},
  \end{equation}
  and the unique solution to
  \begin{equation}\label{weightnorm_ls}
    \min_{x\in\mathcal{X}}\|x\|_{M}, \ \ \ \mathcal{X} = \{x\in\mathbb{R}^n: \|Ax-b\|_2=\min\}
  \end{equation}
  is 
  \begin{equation}
    x_{*} = \sum_{i=1}^{r}\frac{u_{i}^{\top}b}{\sigma_{i}}v_{i}
  \end{equation}
\end{theorem}
\begin{proof}
  Write $U$ and $V$ as $U=(U_r, U_{r,\perp})$ and $V=(V_r, V_{r,\perp})$.
  Using the WSVD of $A$, we have
  \begin{align*}
    \|Ax-b\|_2^2 
    &= 
    \left\|U\begin{pmatrix}
      \Sigma_r & \mathbf{0} \\
      \mathbf{0} & \mathbf{0}
    \end{pmatrix}\begin{pmatrix}
      V_{r}^{\top} \\ V_{r,\perp}^{\top}
    \end{pmatrix}
    Mx-b \right\|_{2}^2
    = \left\|\begin{pmatrix}
      \Sigma_r V_{r}^{\top}Mx \\ \mathbf{0}
    \end{pmatrix} - 
    \begin{pmatrix}
      U_{r}^{\top}b \\  U_{r,\perp}^{\top}b
    \end{pmatrix}
    \right\|_{2}^2 \\
    &= \|\Sigma_r V_{r}^{\top}Mx-U_{r}^{\top}b\|_{2}^2 + \|U_{r,\perp}^{\top}b\|_{2}^2 .
  \end{align*}
  Therefore, the minimizers of $\min_{x\in\mathbb{R}^n}\|Ax-b\|_2$ are the solutions to $\Sigma_r V_{r}^{\top}Mx=U_{r}^{\top}b$, which is equivalent to $V_{r}^{\top}Mx=\Sigma_r^{-1}U_{r}^{\top}b$. An obvious solution to the above equation is $x_{*}=V_{r}\Sigma_r^{-1}U_{r}^{\top}b=\sum_{i=1}^{r}\frac{u_{i}^{\top}b}{\sigma_{i}}v_{i}$. Since $\mathcal{N}(V_{r}^{\top}M)=\mathrm{span}\{v_{r+1},\dots,v_n\}$, we have the expression of $\mathcal{X}$ as \eqref{sol_set}

  On the other hand, for any $x\in\mathcal{X}$ sucht that $x=x_{*}+\sum_{r+1}^{n}\gamma_{i}v_i$, since $v_i$ are mutual $M$-orthogonal, we have
  \begin{align*}
    \|x\|_{M}^2 = \|x_{*}\|_{M}^2 + \sum_{i=r+1}^{n}\gamma_{i}^2 \geq \|x_{*}\|_{M}^2,
  \end{align*}
  where ``='' holds if and only if $\gamma_{r+1}=\cdots=\gamma_{n}=0$. Therefore \eqref{weightnorm_ls} has the unique solution $x_{*}$.
\end{proof}

If we let $A^{\dag_{M}}=V\Sigma^{\dag}U^{\top}$, where $\Sigma^{\dag}:=\begin{pmatrix}
  \Sigma_r^{-1} & \mathbf{0} \\
      \mathbf{0} & \mathbf{0}
\end{pmatrix}\in\mathbb{R}^{n\times m}$. Then we can express $x_{*}$ as $x_{*}=A^{\dag_{M}}b$. This is a similar expression to the smallest 2-norm solution to $\min_{x\in\mathbb{R}^n}\|Ax-b\|_2$. But unfortunately, $A^{\dag_{M}}$ is not a real pseudoinverse of a matrix since $(AA^{\dag_{M}})^{\top}\neq AA^{\dag_{M}}$. Thus, we do not discuss it anymore.

\section{Iterative algorithm for WSVD and applications}\label{sec4}
The WSVD of $A$ is actually the singular value expansion of the linear compact operator $\mathcal{A}: (\mathbb{R}^{n}, \langle\cdot,\cdot\rangle_{M}) \to (\mathbb{R}^{m}, \langle\cdot,\cdot\rangle_{2})$ that has a finite rank. This motivates us to apply the GKB process to $\mathcal{A}$ to approximate several dominant WSVD components of $A$; see \cite{caruso2019convergence} for the GKB process for compact linear operators. 

Starting from a nonzero vector $b\in\mathbb{R}^{m}$\footnote{For the GKB process used in LSQR, it usually uses the right-hand side $b$ as the starting vector. However, for using the GKB process to calculate SVD or WSVD, it can use any nonzero vector in $\mathbb{R}^{m}$ as a starting vector.}, the GKB process proceeds based on the following recursive relations:
\begin{align}
	& \beta_1 p_1 = b, \label{GKB1} \\
	& \alpha_{i}q_i = \mathcal{A}^{*}p_i -\beta_i q_{i-1}, \label{GKB2} \\
	& \beta_{i+1}p_{i+1} = \mathcal{A}q_{i} - \alpha_i p_i, \label{GKB3}
\end{align}
where $\mathcal{A}$ is the adjoint of $\mathcal{A}$. The iteration proceeds as $i=1,2,\dots$, and we set $q_{0}:=\mathbf{0}$. From the definition of $\mathcal{A}$ we have $\mathcal{A}q_i=Aq_i$. In order to compute $\mathcal{A}^{*}$, we use the basic relation
\begin{equation*}
  \langle \mathcal{A}x, y\rangle_{2} = \langle x, \mathcal{A}^{*}y\rangle_{M}
\end{equation*}
which is equivalent to $(Ax)^{\top}y=x^{\top}\mathcal{A}^{*}My$ for any vectors $x\in\mathbb{R}^{n}$ and $y\in\mathbb{R}^{m}$. If follows that $\mathcal{A}^{*}=M^{-1}A^{T}$. Therefore, we obtain the GKB process of $\mathcal{A}^{*}$, as summarized in Algorithm \ref{alg1}. We name it as the weighted GKB process with weight matrix $M$.

\begin{algorithm}[htb]
	\caption{The $k$-step weighted GKB (WGKB)}\label{alg1}
	\begin{algorithmic}
    \Require Matrix $A\in\mathbb{R}^{m\times n}$, positive definite $M\in\mathbb{R}^{n\times n}$, nonzero $b\in\mathbb{R}^{m}$
    \Ensure $\{\alpha_i, \beta_i\}_{i=1}^{k+1}$, $\{p_i, q_i\}_{i=1}^{k+1}$
		\State Let $\beta_{1}=\|b\|_2$,  $p_{1}=b/\beta_{1}$ 
    \State Compute $\bar{s}=A^{\top}p_1$,  $s=M^{-1}\bar{s}$
    \State Compute $\alpha_1=(s^{\top}\bar{s})^{1/2}$, $q_1=s/\alpha_1$
		\For {$i=1,2,\dots,k$}
		\State $r=Aq_{i}-\alpha_{i}p_{i}$ 
		\State $\beta_{i+1}=\|r\|$,  $p_{i+1}=r/\beta_{i+1}$
		\State $\bar{s}=A^{\top}p_{i+1}-\beta_{i+1}Mq_{i}$, $s=M^{-1}\bar{s}$
    \State $\alpha_{i+1}=(s^{\top}\bar{s})^{1/2}$, $q_{i+1}=s/\alpha_{i+1}$
		\EndFor
	\end{algorithmic}
\end{algorithm}

Using the property of GKB for the compact operator $\mathcal{A}$, before the WGKB reaches the termination step, that is, $k_t:=\max_{i}\{\alpha_i\beta_i>0\}$, the $k$-step WGKB process generate two groups of vectors $\{p_1,\dots,p_{k+1}\}$ and $\{q_1,\dots,q_{k+1}\}$ that are 2- and $M$-orthornormal, respectively. If we let $P_{k+1}=(p_1,\dots,p_{k+1})$, $Q_{k}=(p_1,\dots,q_{k})$ and
\begin{equation}
	B_{k}=\begin{pmatrix}
		\alpha_{1} & & & \\
		\beta_{2} &\alpha_{2} & & \\
		&\beta_{3} &\ddots & \\
		& &\ddots &\alpha_{k} \\
		& & &\beta_{k+1}
		\end{pmatrix}\in  \mathbb{R}^{(k+1)\times k},
\end{equation}
then we have 
\begin{align}
  AQ_{k} &= P_{k+1}B_{k}, \label{gkb1} \\
  M^{-1}A^{\top}P_{k+1} &= Q_{k}B_{k}^{\top}+\alpha_{k+1}q_{k}e_{k+1}^{\top} , \label{gkb2}
\end{align}
where $e_{k+1}$ is the $(k+1)$-th column of $I_{k+1}$.
Therefore, $B_k$ is the projection of $A$ onto subspaces $\mathrm{span}\{P_{k+1}\}$ and $\mathrm{span}\{Q_{k}\}$. 

We can expect to approximate several dominant WSVD components of $A$ by the SVD of $B_k$. Let the compact SVD of $B_k$ be
\begin{equation}
  B_k = Y_{k}\Theta_kH_{k}^{\top}, \ \ \ 
  \Theta_k = \mathrm{diag}\left(\theta_{1}^{(k)},\dots,\theta_{k}^{(k)}\right), \ \ \
  \theta_{i}^{(k)}>\dots>\theta_{k}^{(k)}>0 ,
\end{equation}
where $Y_{k}=\left(y_{1}^{(k)},\dots,y_{k}^{(k)}\right)\in\mathbb{R}^{(k+1)\times k}$ and $H_{k}=\left(h_{1}^{(k)},\dots,h_{k}^{(k)}\right)\in\mathbb{R}^{k\times k}$ are two orthornormal matrices. Then the approximation to the WSVD triplet $\left(\sigma_{i}, u_i, v_i\right)$ is $\left(\bar{\sigma}_{i}^{(k)},\bar{u}_{i}^{(k)},\bar{v}_{i}^{(k)}\right):=\left(\theta_{i}^{(k)}, P_{k+1}y_{i}^{(k)},Q_{k}h_{i}^{(k)}\right)$.

\begin{proposition}
  The approximated WSVD triplet based on WGKB satisfies
  \begin{align}
    A\bar{v}_{i}^{(k)}-\bar{\sigma}_{i}^{(k)}\bar{u}_{i}^{(k)} &= 0 , \\
    A^{\top}\bar{u}_{i}^{(k)}- \bar{\sigma}_{i}^{(k)}M\bar{v}_{i}^{(k)} &= \alpha_{k+1}Mq_{k+1}e_{k+1}^{\top} .
  \end{align}
\end{proposition}
\begin{proof}
  These two relations can be verified by directly using \eqref{gkb1} and \eqref{gkb2}:
  \begin{equation*}
    A\bar{v}_{i}^{(k)}-\bar{\sigma}_{i}^{(k)}\bar{u}_{i}^{(k)} 
    = AQ_{k}h_{i}^{(k)} - \theta_{i}^{(k)}P_{k+1}y_{i}^{(k)}
    = P_{k+1}\left(B_{k}h_{i}^{(k)}-\theta_{i}^{(k)}y_{i}^{(k)} \right) = 0
  \end{equation*}
  and 
  \begin{align*}
    A^{\top}\bar{u}_{i}^{(k)}- \bar{\sigma}_{i}^{(k)}M\bar{v}_{i}^{(k)}
    &= A^{\top}P_{k+1}y_{i}^{(k)}-\theta_{i}^{(k)}MQ_{k}h_{i}^{(k)}
    = M\left(Q_{k}B_{k}^{\top}+\alpha_{k+1}q_{k}e_{k+1}^{\top}\right)y_{i}^{(k)} -\theta_{i}^{(k)}MQ_{k}h_{i}^{(k)} \\
    &= MQ_k(B_{k}^{\top}y_{i}^{(k)}-\theta_{i}^{(k)}h_{i}^{(k)}) + \alpha_{k+1}Mq_{k}e_{k+1}^{\top}y_{i}^{(k)} \\
    &= \alpha_{k+1}Mq_{k}e_{k+1}^{\top}y_{i}^{(k)} .
  \end{align*}  
  The proof is completed.
\end{proof}

Therefore, the triplet $\left(\bar{\sigma}_{i}^{(k)},\bar{u}_{i}^{(k)},\bar{v}_{i}^{(k)}\right)$ can be accepted as a satisfied WSVD triplet at the iteration that $\left|\alpha_{k+1}q_{k}e_{k+1}^{\top}y_{i}^{(k)}\right|$ is sufficiently small. This easily computed quantity can be used as a stopping criterion for iteratively computing WSVD triplets.

To solve the large-scale least square problem \eqref{weightnorm_ls}, one method is to transform it to the standard one:
\begin{equation*}
  \min_{z\in\mathcal{Z}}\|z\|_2, \ \ \ 
  \mathcal{Z} = \{z\in\mathbb{R}^{n}: \|AL_{M}^{-1}z-b\|_2\}
\end{equation*}
by the substitution $z=L_{M}x$, where $L_M$ is the Cholesky factor of $M$, i.e. $M=L_{M}^{\top}L_{M}$, and then use the LSQR algorithm to solve it. However, this transformation needs to compute the Cholesky factorization of $M$ in advance, which can be very costly for large-scale $M$. Noticing that the least square problem \eqref{weightnorm_ls} can be obtained by WSVD, that is $x_{*}=A^{\dag_M}b$, we can expect to iteratively compute $x_{*}$ based on the WGKB process of $A$ with starting vector $b$. Note from \eqref{GKB1} that $\beta P_{k+1}e_{1}=b$, where $e_1$ is the first column of $I_{k+1}$. If the WGKB process does not terminate until the $k$-th step, i.e. $k< k_t$, then $B_k$ has full column rank. In this case, we seek a solution to \eqref{weightnorm_ls} in the subspace $\mathrm{span}\{Q_{k}\}$. By letting $x=Q_{k}y$ with $y\in\mathbb{R}^{k}$, we have
\begin{equation*}
  \min_{x\in\mathrm{span}\{Q_{k}\}}\|Ax-b\|_2
  = \min_{y\in\mathbb{R}^{n}}\|AQ_ky-b\|_2
  = \min_{y\in\mathbb{R}^{n}}\|P_{k+1}B_{k}y-\beta P_{k+1}e_{1}\|_2
  = \min_{y\in\mathbb{R}^{n}}\|B_{k}y-\beta e_{1}\|_2
\end{equation*}
and $\|x\|_{M} = \|Q_{k}y\|_{M} = \|y\|_2$.
Therefore, the problem \eqref{weightnorm_ls} with $x\in\mathrm{span}\{Q_{k}\}$ becomes
\begin{equation}
  \min_{y\in\mathcal{Y}_k}\|y\|_2, \ \ \ 
  \mathcal{Y}_k = \{y\in\mathbb{R}^{k}: \|B_{k}y-\beta e_{1}\|_2=\min \}
\end{equation}
This is a standard linear least squares problem with minimum 2-norm, which has the unique solution $y_k=B_k^{\dag}\beta e_{1}$. Therefore, at the $k$-th iteration, we compute the iterative approximation to \eqref{weightnorm_ls}:
\begin{equation}\label{x_k}
  x_k = Q_{k}y_k, \ \ \ y_k=B_k^{\dag}\beta e_{1}.
\end{equation}

The above procedure is similar to the LSQR algorithm for standard 2-norm least squares problems. Moreover, the bidiagonal structure of $B_k$ allows us to design a recursive relation to update $x_k$ step by step without explicitly computing $B_k^{\dag}\beta e_{1}$ at each iteration; see \cite[Section 4.1]{Paige1982} for the similar recursive relation in LSQR. We summarized the iterative algorithm for iteratively solving \eqref{weightnorm_ls} in Algorithm \ref{alg2}, which is named the weighted LSQR (WLSQR) algorithm.

\begin{algorithm}[htb]
	\caption{Weighted LSQR (WLSQR)}\label{alg2}
	\begin{algorithmic}
    \Require Matrix $A\in\mathbb{R}^{m\times n}$, positive definite $M\in\mathbb{R}^{n\times n}$, vector $b\in\mathbb{R}^{m}$
    \Ensure Approximate solution to \eqref{weightnorm_ls}: $x_k$
		\State Compute $\beta_1p_1=b$, $\alpha_1q_1=M^{-1}A^{T}q_1$, 
    \State Set $x_0=\mathbf{0}$, $w_1=q_1$, $\bar{\phi}_{1}=\beta_1$, $\bar{\rho}_1=\alpha_1$
		\For {$i=1,2,\dots$ until convergence,}
    \State \textbf{(Applying the WGKB process)}
		\State $\beta_{i+1}p_{i+1}=Aq_{i}-\alpha_ip_{i}$ 
		\State $\alpha_{i+1}q_{i+1}=M^{-1}A^{\top}p_{i+1}-\beta_{i+1}q_{i}$
    \State \textbf{(Applying the Givens QR factorization to $B_k$)}
		\State $\rho_{i}=(\bar{\rho}_{i}^{2}+\beta_{i+1}^{2})^{1/2}$
		\State $c_{i}=\bar{\rho}_{i}/\rho_{i}$
    \State $s_{i}=\beta_{i+1}/\rho_{i}$
		\State$\theta_{i+1}=s_{i}\alpha_{i+1}$ 
    \State $\bar{\rho}_{i+1}=-c_{i}\alpha_{i+1}$
		\State $\phi_{i}=c_{i}\bar{\phi}_{i}$ 
    \State $\bar{\phi}_{i+1}=s_{i}\bar{\phi}_{i}$
    \State \textbf{(Updating the solution)}
		\State $x_{i}=x_{i-1}+(\phi_{i}/\rho_{i})w_{i} $
		\State $w_{i+1}=v_{i+1}-(\theta_{i+1}/\rho_{i})w_{i}$
		\EndFor
	\end{algorithmic}
\end{algorithm}

The following result shows that the WLSQR algorithm approaches the exact solution to \eqref{weightnorm_ls} as the algorithm proceeds.

\begin{theorem}\label{naive}
  If the WGKB process terminates at step $k_t=\max_{i}\{\alpha_{i}\beta_{i}>0\}$, then the iterative solution $x_{k_t}$ is the exact solution to \eqref{weightnorm_ls}.
\end{theorem}
  \begin{proof}
    By Theorem \ref{thm6}, a vector $x\in\mathbb{R}^{n}$ is the unique solution to \eqref{weightnorm_ls} if and only if
    \begin{equation*}
      Ax-b \perp \mathcal{R}(A), \ \ \ 
      x \perp_{M} \mathrm{span}\{v_{r+1},\dots,v_{n}\} .
    \end{equation*}
    Using the property of the GKB process of $\mathcal{A}$, the subspace $\mathrm{span}\{Q_k\}$ can be expressed as the Krylov subspace
    \begin{equation*}
      \mathrm{span}\{Q_k\} = \mathcal{K}_{k}(\mathcal{A}^{*}\mathcal{A},\mathcal{A}^{*}b)=\mathrm{span}\{(\mathcal{A}^{*}\mathcal{A})^{i}\mathcal{A}^{*}b\}_{i=0}^{k-1} = \mathrm{span}\{(M^{-1}A^{\top}A)^{i}M^{-1}A^{\top}b\}_{i=0}^{k-1}.
    \end{equation*} 
    For any $k\leq k_t$, since $x_k=Q_{k}y_k$, it follows
    \begin{align*}
      x_{k} \in \mathrm{span}\{(M^{-1}A^{\top}A)^{i}M^{-1}A^{\top}b\}_{i=0}^{k-1} \subseteq \mathcal{R}(M^{-1}A^{\top})
      = M^{-1}\mathcal{N}(A)^{\perp}.
    \end{align*}
    Using the WSVD of $A$, we have $\mathcal{N}(A)=\mathrm{span}\{v_{r+1},\dots,v_{n}\}$. For any $v\in\mathbb{R}^{n}$, it follows that 
    \begin{equation*}
      v\in M^{-1}\mathcal{N}(A)^{\perp} \Leftrightarrow Mv \in \mathcal{N}(A)^{\perp}
      \Leftrightarrow v^{\top}Mv_{i} = 0, \ i=r+1,\dots,n,
    \end{equation*}
    which leads to $M^{-1}\mathcal{N}(A)^{\perp}=\mathrm{span}\{v_{1},\dots,v_{r}\}$. Therefore, we get $x_k\in\mathrm{span}\{v_{1},\dots,v_{r}\}$ and thereby $x_k\perp_{M} \mathrm{span}\{v_{r+1},\dots,v_{n}\}$.
  
    To prove $Ax_{k_t}-b \perp \mathcal{R}(A)$, we only need to show $A^{\top}(Ax_{k_t}-b)=\mathbf{0}$. By \eqref{GKB2} and \eqref{GKB3}, we have
    \begin{align*}
      A^{\top}(Ax_{k_t}-b) 
      &= A^{\top}(AQ_{k_t}y_{k_t}-P_{k_t+1}\beta_{1}e_1) \\
      &= A^{\top}P_{k_t+1}(B_{k_t}y_{k_t}-\beta_{1}e_1) \\
      &= M(Q_{k_t}B_{k_t}^{\top}+\alpha_{k_t+1}q_{k_t+1}e_{k+1}^{\top})(B_{k_t}y_{k_t}-\beta_{1}e_1) \\
      &= M\left[Q_{k_t}(B_{k_t}^{\top}B_{k_t}y_{k_t}- B_{k_t}^{\top}\beta_{1}e_{1}) + \alpha_{k_t+1}\beta_{k_t+1}q_{k+1}e_{k_t}^{\top}y_{k_t}\right] \\
      &= \alpha_{k_t+1}\beta_{k_t+1}Mq_{k_t+1}e_{k_t}^{\top}y_{k_t},
    \end{align*}
    where we used $B_{k_t}^{\top}B_{k_t}y_{k_t}= B_{k_t}^{\top}\beta_{1}e_{1}$ since $y_{k_t}$ satisfies the normal equation of $\min_{y}\|B_{k_t}y-\beta_{1}e_1\|_2$. Since WGKB terminates at $k_t$, which means that $\alpha_{k_t+1}\beta_{k_t+1} = 0$, we have $A^{\top}(Ax_{k_t}-b)=\mathbf{0}$. 
  \end{proof}

\section{Using WSVD to analyze and solve linear ill-posed problems}\label{sec5}
For the Tikhonov regularization \eqref{tikh1} with the $\|x\|_{M}^2$ regularization term, if we have the Cholesky factorization $M=L_{M}^{\top}L_{M}$, this problem can be transformed into the standard-form Tikhonov regularization problems
$\min_{\bar{x}\in\mathbb{R}^{n}}\{\|AL_{M}^{-1}\bar{x}-b\|_{2}^2 + \lambda\|\bar{x}\|_{2}^2 \}$ by letting $\bar{x}=L_{M}x$. To avoid computing the Cholesky factorization of $M$, we can write its solution explicitly using the WSVD of $A$. 
\begin{theorem}
  The solution to the Tikhonov regularization \eqref{tikh1} can be written as
  \begin{equation}\label{tikh1_sol}
    x_{\lambda} =\sum_{i=1}^{r}\frac{\sigma_{i}^2}{\sigma_{i}^2+\lambda}\frac{u_{i}^{\top}b}{\sigma_i}v_i .
  \end{equation}
\end{theorem}
\begin{proof}
  Since $V$ is an $M$-orthogonal matrix, we can let $x=Vy$ for any $x\in\mathbb{R}^{n}$ where $y\in\mathbb{R}^{n}$. Using relations $AV=U\Sigma$ and $V^{\top}MV=I_{n}$ in the WSVD of $A$, \eqref{tikh1} becomes 
\begin{equation*}
  \min_{y\in\mathbb{R}^n}\{\|U\Sigma y-b\|_{2}^2 + \lambda \|y\|_{2}^2\} .
\end{equation*}
The normal equation of this problem is
\begin{equation*}
  \left[(U\Sigma)^{\top}(U\Sigma)+\lambda I_{n}\right]y = (U\Sigma)^{\top}b \ \ \Leftrightarrow \ \ 
  (\Sigma^{\top}\Sigma+\lambda I_n)y = \Sigma^{\top}U^{\top}b,
\end{equation*}
which leads to the unique solution to \eqref{tikh1} as \eqref{tikh1_sol}.
\[x_{\lambda} = V(\Sigma^{\top}\Sigma+\lambda I_n)^{-1}\Sigma^{\top}U^{\top}b
= \sum_{i=1}^{r}\frac{\sigma_{i}^2}{\sigma_{i}^2+\lambda}\frac{u_{i}^{\top}b}{\sigma_i}v_i. \]
The proof is completed.
\end{proof}

The above expression of $x_{\lambda}$ is similar to \eqref{fil_sol}, where $\lambda$ should be chosen properly to filter out the noisy components.

To avoid computing the Cholesky factorization of $M$ and choosing a proper regularization parameter in advance, we consider the \textit{subspace projection regularization} (SPR) method following the idea in \cite[\S 3.3]{Engl2000}, which can be formed as
\begin{equation}\label{spr}
	\min_{x\in\bar{\mathcal{X}}_k}\|x\|_{M}, \ \ \bar{\mathcal{X}}_k = \{x: \min_{x\in\mathcal{S}_k}\|Ax-b\|_{2}\}.
\end{equation}

\begin{rem}
  The above SPR method is a generalization of the iterative regularization method corresponding to the $\|x\|_{2}^2$ regularization term. For example, the LSQR method with early stopping can be written as
  \begin{equation*}
    \min_{x\in\bar{\mathcal{X}}_k}\|x\|_{2}, \ \ \ \bar{\mathcal{X}}_k = \{x: \min_{x\in\mathcal{S}_k}\|Ax-b\|_{2}\}, \ \ 
    \mathcal{S}_k=\mathcal{K}_{k}(A^{\top}A,A^{\top}b) .
  \end{equation*}
  The success of the SPR method highly depends on the choice of solution subspaces $\mathcal{S}_k$, which should be elaborately constructed to incorporate the prior information about the desired solution encoded by the regularizer $\|x\|_{M}^2$. For the LSQR method, the solution subspaces $\mathcal{K}_{k}(A^{\top}A,A^{\top}b)$ can only deal with the $\|x\|_{2}^2$ regularization term. This motivates us to develop a new iterative process to construct solution subspaces to incorporate prior information encoded by $\|x\|_{M}^2$.
\end{rem}

\begin{rem}
  For any choice of a $k$-dimensional $\mathcal{S}_k$, there exists a unique solution to \eqref{spr}. To see it, let $x=W_ky$ with $y\in\mathbb{R}^{k}$ be any vector in $\mathcal{S}_k$, where $W_k\in\mathbb{R}^{n\times k}$ whose columns are $M$-orthonormal and span $\mathcal{S}_k$. Then the solution to \eqref{spr} satisfies $x_k=W_ky_k$, where $y_k$ is the solution to
  \begin{equation*}
    \min_{y\in\mathcal{Y}_k}\|y\|_{2}, \ \ \mathcal{Y}_k = \{y: \min_{y\in\mathbb{R}^{k}}\|AW_ky-b\|_{2}\} .
  \end{equation*}
This problem has a unique solution $y_k=(AW_k)^{\dag}b$. Therefore, there exists a unique solution to \eqref{spr}.
\end{rem}

Using the WSVD of $A$, if we choose the $k$-th solution subspace in \eqref{spr} as $\mathcal{S}_k=\mathrm{span}\{v_1,\dots,v_k\}$, then the solution to \eqref{spr} is $x_k=V_ky_k$, where $y_k$ is the solution to
\begin{equation*}
  \min_{y\in\mathcal{Y}_k}\|y\|_{2}, \ \ \mathcal{Y}_k = \{y: \min_{y\in\mathbb{R}^{k}}\|U_k\Sigma_ky-b\|_{2}\} .
\end{equation*}
Note that $U_k\Sigma_k$ has full column rank for $1\leq k \leq r$. Therefore, $\mathcal{Y}_k$ has only one element, which is $y_k=\Sigma_{k}^{-1}U_{k}^{\top}b$, thereby 
$$x_k=V_k\Sigma_{k}^{-1}U_{k}^{\top}b=\sum_{i=1}^{k}\frac{u_{i}^{\top}b}{\sigma_{i}}v_i.$$
Comparing this result with Theorem \ref{thm6}, we find that $x_k$ can be obtained by truncating the first $k$ components of $x_{*}$. Thus, we name this form of $x_k$ as the truncated WSVD (TWSVD) solution. By truncating the above solution at a proper $k$, the TWSVD solution can capture the main information corresponding to the dominant right weighted singular vectors $v_i$ while discarding the highly amplified noisy vectors corresponding to others. 

From the above investigation, those dominant $v_i$ play an important role in the regularized solution, since they contain the desirable information about $x$ encoded by the regularizer $\|x\|_{M}^2$. As has been shown that WGKB can be used to approximate the WSVD triplets of $A$, this motivates us to design iterative regularization algorithms based on WGKB. This can be achieved by setting the $k$-th solution subspace as $\mathcal{S}_k=\mathrm{span}\{Q_k\}$. Following the same procedure for deriving WLSQR, the problem \eqref{spr} becomes
\begin{equation*}
  \min_{y\in\mathcal{Y}_k}\|y\|_{2}, \ \ \mathcal{Y}_k = \{y: \min_{y\in\mathbb{R}^{k}}\|B_ky-\beta_1e_1\|_{2}\} ,
\end{equation*}
which has the solution given by \eqref{x_k}. Therefore, the SPR method with $\mathcal{S}_k=\mathrm{span}\{Q_k\}$ is actually the WLSQR method. One important difference from solving the ordinary least squares problem is that here the iteration should be early stopped. This can be seen from Theorem \ref{naive} since the algorithm eventually converges to the naive solution to \eqref{weightnorm_ls}. This is the typical \textit{semi-convergence} behavior for SPR methods: as the iteration proceeds, the iterative solution first gradually approximate to $x_{\text{true}}$, then the solution will deviate far from $x_{\text{true}}$ and eventually converges to $A^{\dag_M}b$ \cite[\S 3.3]{Engl2000}. This is because the solution subspace will contain more and more noisy components as it gradually expands. The iteration at which the corresponding solution has the smallest error is called the semi-convergence point. Note that the iteration number $k$ in SPR plays a similar role as the regularization parameter in Tikhonov regularization. Here we adapt two criteria for choosing regularization parameters to estimate the semi-convergence point.

\paragraph*{Two early stopping rules}
\begin{enumerate}
  \item For the Gaussian white noise $e$, if an estimate of $\|e\|_2$ is known, one criterion for determining the early stopping iteration is the discrepancy principle (DP), which states that the discrepancy between the data and predicted output, which is $\|Ax_k-b\|_2$, should be of the order of $\|e\|_2$ \cite{Morozov1966}. Following the derivation of the procedure for updating $x_k$ in \cite{Paige1982}, we have 
  $$\|Ax_k-b\|_2=\|B_ky_k-\beta_1e_1\|=\bar{\phi}_{k+1}.$$ 
  Note that $\|Ax_k-b\|_2$ decreases monotonically since $x_k$ minimizes $\|Ax_k-b\|_2$ in the gradually expanding subspace $\mathcal{S}_k$. Therefore, following DP, we should stop the iteration at the first $k$ that satisfies
  \begin{equation}\label{discrepancy}
    \bar{\phi}_{k+1} \leq \tau\|e\|_2 < \bar{\phi}_{k}
  \end{equation}
  with $\tau>1$ slightly, such as $\tau=1.01$. Typically, the early stopping iteration determined by DP is slightly smaller than the semi-convergence point, thereby the corresponding solution is slightly over-smoothed.
  \item The L-curve (LC) criterion is another early stopping rule that does not need an estimate on $\|e\|_2$ \cite{Hansen1992}. In this method, the log-log scale the curve
  \begin{equation}
    \left(\log\|Ax_k-b\|_{2},\log\|x_k\|_{M}\right) 
    = \left(\log\bar{\phi}_{k+1},\log(\|x_k\|_{M})\right),
  \end{equation}
  is plotted, which usually has a characteristic `L' shape. The iteration corresponding to the corner of the L-curve, which is defined as the point of maximum curvature of the L-curve in a log-log plot, is usually a good early stopping iteration.
\end{enumerate}

The whole process of the WGKB-based SPR iterative algorithm is summarized in Algorithm \ref{alg3}.

\begin{algorithm}[htb]
	\caption{WLSQR with early stopping}\label{alg3}
	\begin{algorithmic}[1]
		\Require Matrix $A\in\mathbb{R}^{m\times n}$, positive definite $M\in\mathbb{R}^{n\times n}$, vector $b\in\mathbb{R}^{m}$
    \Ensure Final regularized solution corresponding to \eqref{spr}
		\For{$k=1,2,\ldots,$}
    \State Compute $\alpha_k$, $\beta_k$, $p_k$, $q_k$ by WGKB
		\State Update $x_k$ from $x_{k-1}$
    \State Compute the $\|Ax_k-b\|_2$ and $\|x\|_{M}$
		\If{Early stopping criterion is satisfied}   \Comment{DP or LC}
		\State Estimate the semi-convergence point as $k_1$
		\State Terminate the iteration to get $x_{k_1}$
		\EndIf
		\EndFor
	\end{algorithmic}
\end{algorithm}

To show the effectiveness of WLSQR for regularizing \eqref{ill-posed1} with regularizer $\|x\|_{M}^2$, we give the following result.

\begin{theorem}\label{thm5.1}
  Let the Cholesky factorization of $M$ be $M=L_{M}^{\top}L_{M}$. Then the $k$-th iterative solution of WLSQR for \eqref{spr} is $x_k=L_{M}^{-1}\bar{x}_k$, where $\bar{x}_k$ is the $k$-th LSQR solution of $\min_{\bar{x}\in\mathbb{R}^{n}}\|AL_{M}^{-1}\bar{x}-b\|_2$.
\end{theorem}
\begin{proof}
  The $k$-th LSQR solution of $\min_{\bar{x}\in\mathbb{R}^{n}}\|AL_{M}^{-1}\bar{x}-b\|_2$ is the solution of the subspace constrained least squares problem
  \begin{equation*}
    \min_{\bar{x}\in\bar{\mathcal{S}}_k}\|AL_{M}^{-1}\bar{x}-b\|_2, \ \ \ 
    \bar{\mathcal{S}}_k = \mathcal{K}_k\left((AL_{M}^{-1})^{\top}(AL_{M}^{-1}), (AL_{M}^{-1})^{\top}b\right) ,
  \end{equation*}
  where 
  \begin{equation*}
    \mathcal{K}_k\left((AL_{M}^{-1})^{\top}(AL_{M}^{-1}), (AL_{M}^{-1})^{\top}b\right) 
    = \mathrm{span}\left\{(L_{M}^{-\top}A^{\top}AL_{M}^{-1})^{i}L_{M}^{-\top}A^{\top}b\right\}_{i=0}^{k-1}
    = L_{M}^{-\top}\mathrm{span}\left\{(A^{\top}AM^{-1})^{i}A^{\top}b\right\}_{i=0}^{k-1} .
  \end{equation*}
  Therefore, $x_k=L_{M}^{-1}\bar{x}_k$ is the solution of the problem $\min_{x\in L_{M}^{-1}\bar{\mathcal{S}}_k}\|Ax-b\|_2$. From the proof of Proposition \ref{naive}, we have
  \begin{equation*}
    L_{M}^{-1}\bar{\mathcal{S}}_k = M^{-1}\mathrm{span}\left\{(A^{\top}AM^{-1})^{i}A^{\top}b\right\}_{i=0}^{k-1} 
    = \mathrm{span}\{(M^{-1}A^{\top}A)^{i}M^{-1}A^{\top}b\}_{i=0}^{k-1}
    = \mathcal{K}_{k}(\mathcal{A}^{*}\mathcal{A},\mathcal{A}^{*}b),
  \end{equation*}
  which is the $k$-th solution subspace $\mathcal{S}_k$ generated by WGKB. By writting any vector in $\mathcal{S}_k$ as $x=Q_ky$ with $y\in\mathbb{R}^{k}$, it is easy to verify that $\min_{x\in\mathcal{S}_k}\|Ax-b\|_2$ has the unqie solution. It follows that $x_k$ is the $k$-th WLSQR solution of \eqref{spr}.
\end{proof}

From this theorem, we find that WLSQR has the same effect as first transforming the original problem to $\min_{\bar{x}}\|AL_{M}^{-1}\bar{x}-b\|_2$ and then regularizing it. Therefore, this approach makes full use of the information encoded by the regularizer $\|x\|_{M}^2$, due to the elaborately constructed solution subspaces by the WGKB process.

\section{Numerical experiments}\label{sec6}
We consider the Fredholm integral equation of the first kind as \eqref{fred0}. The aim is to recover the unknown $f(t)$ from the noisy observation $g(s)$. We chose the following four examples to perform the numerical experiments.

\paragraph*{Example 1.} 
This example is chosen from \cite{Hansen1998} with the name {\sf shaw}. It models a one-dimensional image restoration problem using the Fredholm integral equation \eqref{fred0}, where the kernel $K$ and solution $f$ are given by
\begin{align*}
	K(s, t) &= (\cos s + \cos t)^2 \left(\frac{\sin u}{u}\right)^2, \ \ \
	u = \pi(\sin s + \sin t) ,  \\
	f(t) &= 2\exp\left(-6(t-0.8)^2\right) + \exp\left(-2(t+0.5)^2\right) .
\end{align*}
where $t\in[-\pi/2, \pi/2]$ and $s\in[-\pi/2,\pi/2]$.

\paragraph*{Example 2.}
This example is Phillips' famous test problem \cite{phillips1962technique}. Define the function
\begin{equation*}
  \phi(x) = 
  \begin{cases}
    1+\cos(\frac{\pi x}{3}), \ \ |x|<3 \\
    0, \ \  |x|\geq 3
  \end{cases}.
\end{equation*}
Then the kernel $K$, the solution $f$ and the exact observation are given by
\begin{align*}
  K(s,t) &= \phi(s-t) , \\
  f(t)   &= \phi(t),  
\end{align*}
where $t\in[-6, 6]$ and $s\in[-6,6]$.

\paragraph*{Example 3.}
This test problem is constructed by ourselves. Define the kernel function and true solution as
\begin{align*}
  K(s,t) &= e^{st}, \\
  f(t)   &= e^t \cos t
\end{align*}
where $t\in[0, 1]$ and $s\in[0,1]$.

\paragraph*{Example 4.}
This test problem is constructed by ourselves. Define the kernel function and true solution as
\begin{align*}
  K(s,t) &=
  \begin{cases}
    s(1-t), \ \ s<t  \\
    t(1-s), \ \ s\geq t
  \end{cases},  \\
  f(t)   &= t-2t^2+t^3 .
\end{align*}
where $t\in[0, 1]$ and $s\in[0,1]$.

To discretize the Fredholm integral equation \eqref{fred0}, we partition the interval $[t_1, t_2]$ into $2l$ uniform subintervals to get $n=2l+1$ grid points $t_1=p_1<p_2<\cdots<p_{n-1}<p_n=t_2$. The whole integral is partitioned as
\begin{equation*}
  \int_{t_1}^{t_2}K(s,t)f(t)\mathrm{d}t =
  \sum_{i=1}^{l}\int_{p_{2i-1}}^{p_{2i+1}}K(s,t)f(t)\mathrm{d}t,
\end{equation*}
where each integral is approximated by Simpson's rule
\begin{equation*}
  \int_{p_{2i-1}}^{p_{2i+1}}K(s,t)f(t)\mathrm{d}t \approx
  \frac{p_{2i+1}-p_{2i-1}}{6}\left[K(s,p_{2i-1})+4K(s,p_{2i})+K(s,p_{2i+1})\right] .
\end{equation*}
Therefore, the whole integral is approximated as
\begin{equation*}
  \int_{t_1}^{t_2}K(s,t)f(t)\mathrm{d}t \approx
  \sum_{i=1}^{n}w_iK(s,p_{i}),
\end{equation*}
with weights
\begin{equation*}
  \frac{h}{3}\{1,4,2,4,2,4,\dots,2,4,1\}, \ \ \ h=(t_2-t_1)/n .
\end{equation*}
The observations are selected from $m$ uniform points in $[s_1, s_2]$ to get an $m$-dimensional vector. The task is to recover the true vector $x_{\text{true}}=(f(p_1),\dots,f(p_n))^{\top}$ from the noisy observation $b$ constructed as follows:
\begin{equation}\label{disc_inv}
  b = Ax +e,
\end{equation}
where $e\in\mathbb{R}^{m}$ is a discrete Gaussian white noise vector, and $A$ is the discretized kernel:
\begin{equation}
  A = \begin{pmatrix} 
     K(s_1,t_1)w_1 &K(s_1,t_2)w_2  & \cdots & K(s_1,t_n)w_n \\
  K(s_2,t_1)w_1 &K(s_2,t_2)w_2  & \cdots & K(s_2,t_n)w_n \\
      \vdots & \vdots &   \ddots & \vdots \\
  K(s_m,t_1)w_1 &K(s_m,t_2)w_2 & \cdots & K(s_m,t_n)w_n \\
      \end{pmatrix} \in \mathbb{R}^{m\times n}.
\end{equation} 
The scale of the noise is controlled by the noise level $\varepsilon:=\|e\|_2/\|Ax_{\text{true}}\|_2$, which may vary for different test examples. The properties of the test examples are shown in Table \ref{tab1}. The discretized true solutions and noisy observations with $\varepsilon=10^{-2}$ are shown in Figure \ref{fig1}.

From the above discretization, the data fidelity term and regularization term for the discrete ill-posed linear system \eqref{disc_inv} should be $\|Ax-b\|_{2}^2$ and $\|x\|_{M}^2$, respectively, where $M$ is the weight matrix 
\begin{equation}
  M = \frac{h}{3}\mathrm{diag}(1,4,2,4,2,4,\dots,2,4,1) .
\end{equation} 

\begin{table}[]
	\centering
	\caption{Properties of the four test examples}
	\scalebox{1.0}{
	\begin{tabular}{lllll}
		\toprule
		Problem     & Example 1  & Example 2 & Example 3  & Example 4  \\
		\midrule
		$m\times n$  & $2500\times 2001$ & $3000\times 2501$ & $3500\times 3001$ & $4000\times 3501$  \\
		Condition number & $5.89\times 10^{18}$ & $2.14\times 10^9$ & $5.98\times 10^{18}$ & $1.27\times 10^{7}$  \\
		\bottomrule
	\end{tabular}}
	\label{tab1}
\end{table}

\begin{figure}[htbp]
	\centering
	\subfloat[]{\label{fig:1a}
  \includegraphics[width=0.25\textwidth]{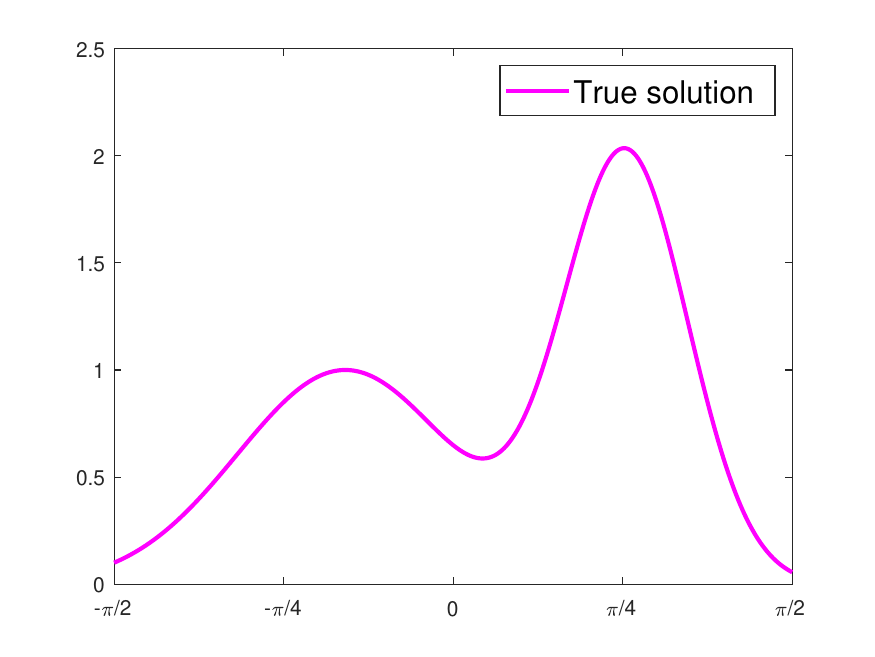}}
	\subfloat[]{\label{fig:1b}
  \includegraphics[width=0.25\textwidth]{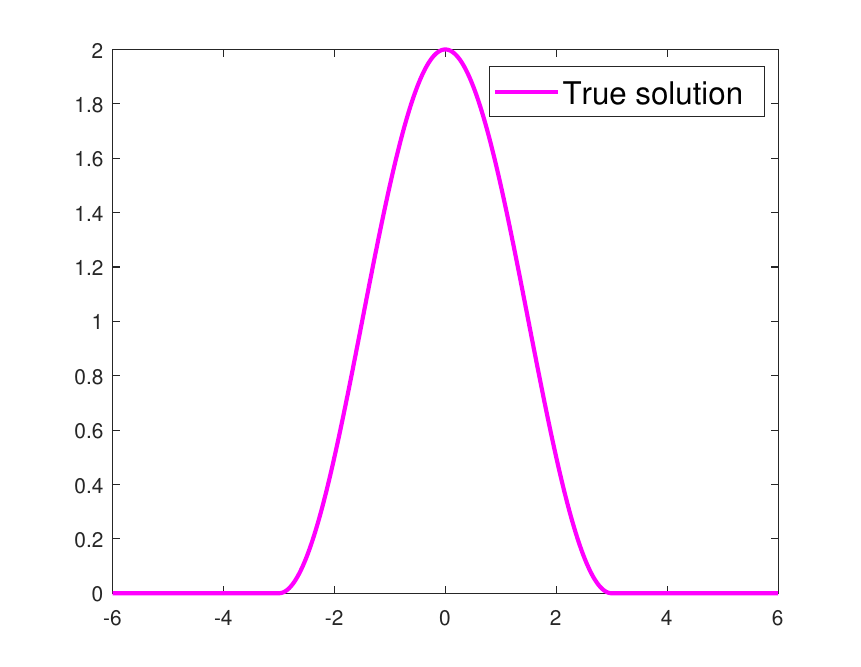}}
	\subfloat[]{\label{fig:1c}
  \includegraphics[width=0.25\textwidth]{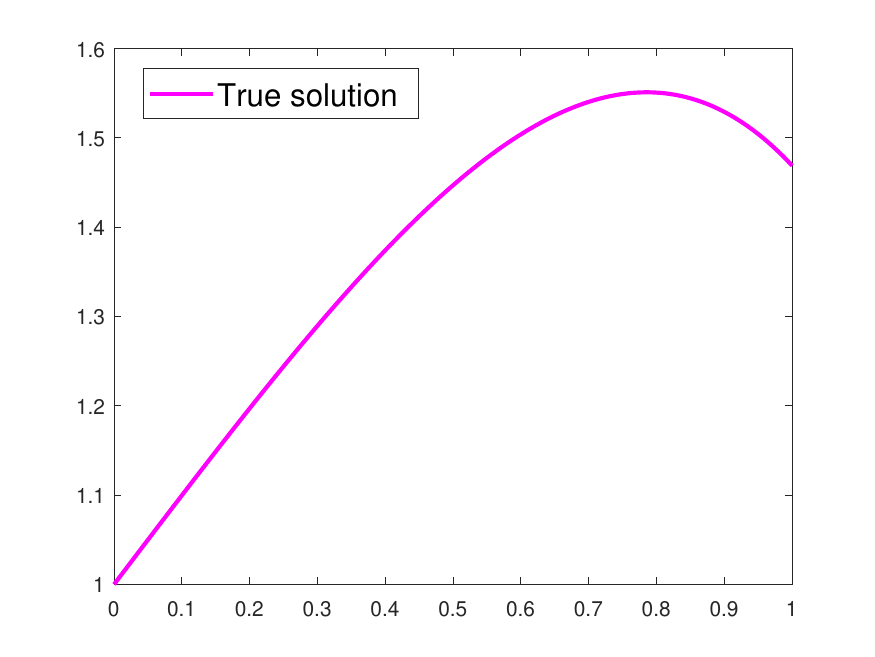}}
	\subfloat[]{\label{fig:1d}
  \includegraphics[width=0.25\textwidth]{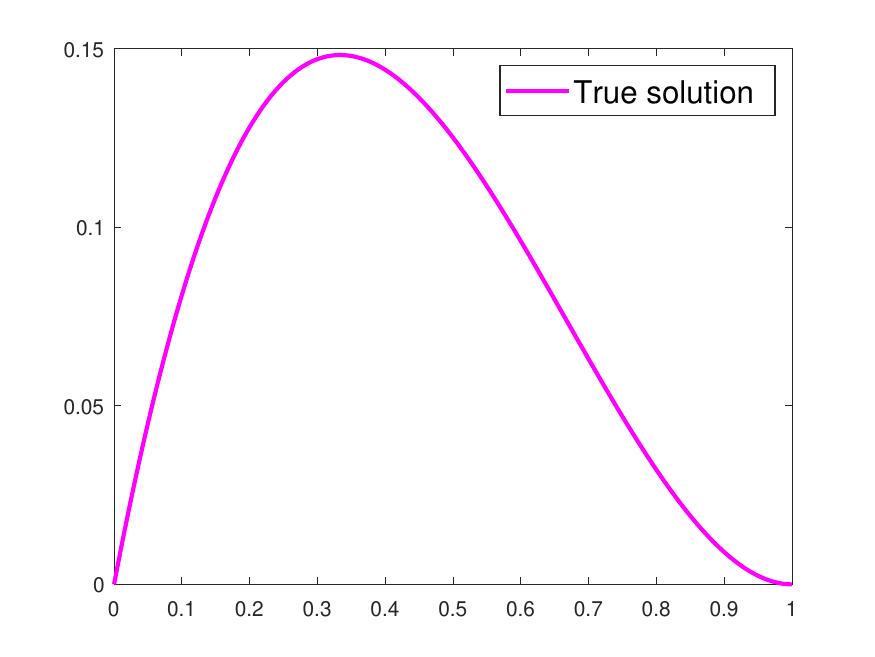}}
  \vspace{-2mm}
	\subfloat[]{\label{fig:1e}
  \includegraphics[width=0.25\textwidth]{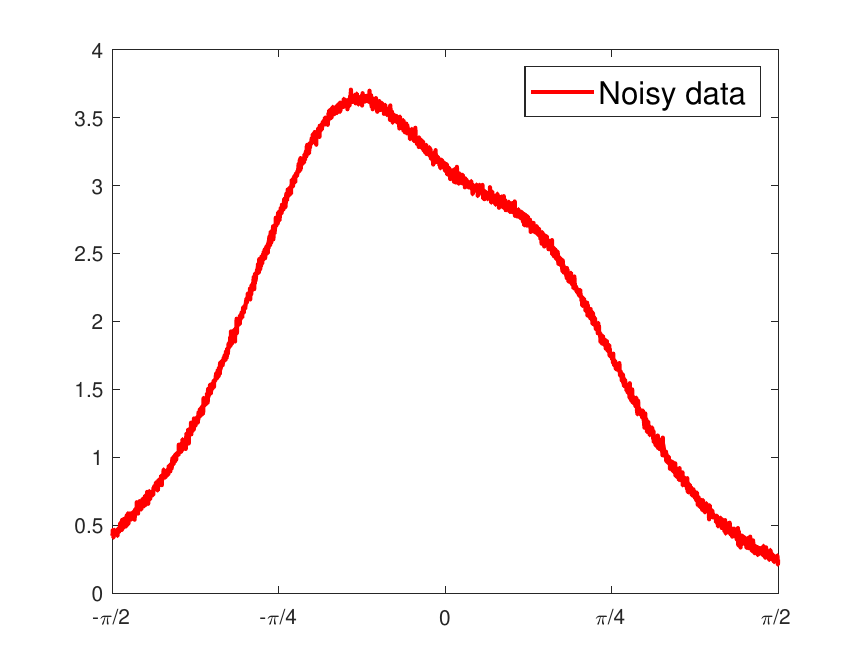}}
  \subfloat[]{\label{fig:1f}
  \includegraphics[width=0.25\textwidth]{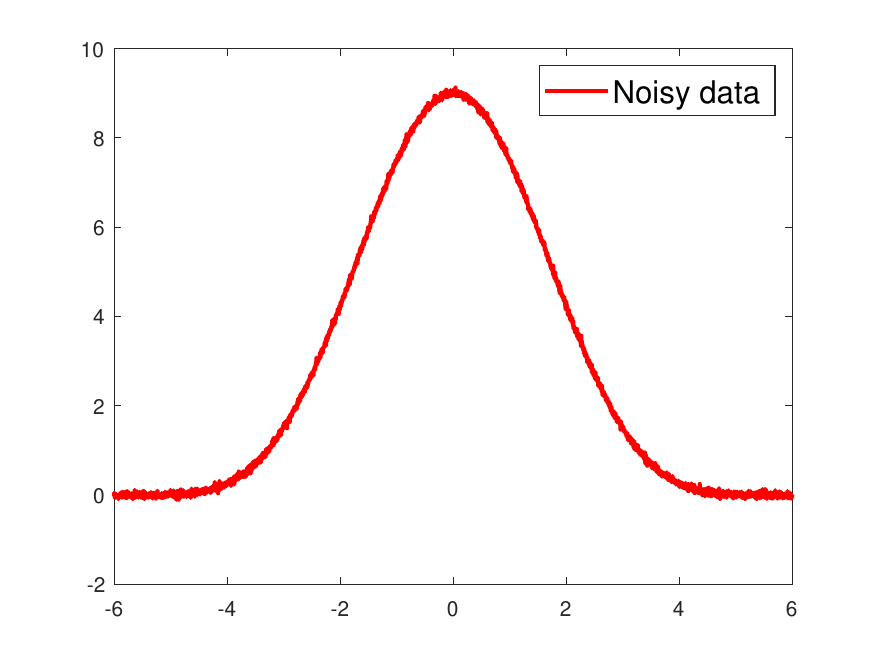}}
	\subfloat[]{\label{fig:1g}
  \includegraphics[width=0.25\textwidth]{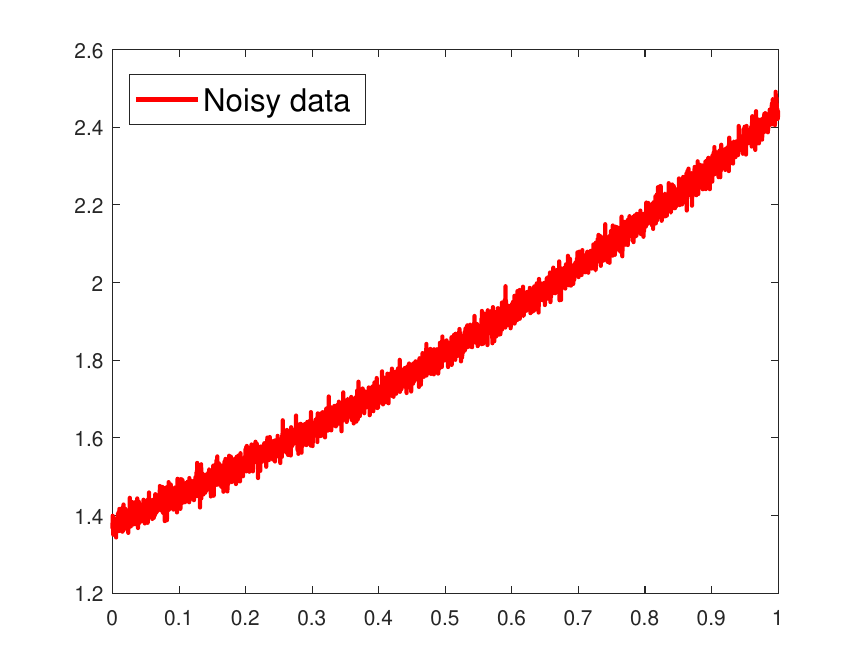}}
	\subfloat[]{\label{fig:1h}
  \includegraphics[width=0.25\textwidth]{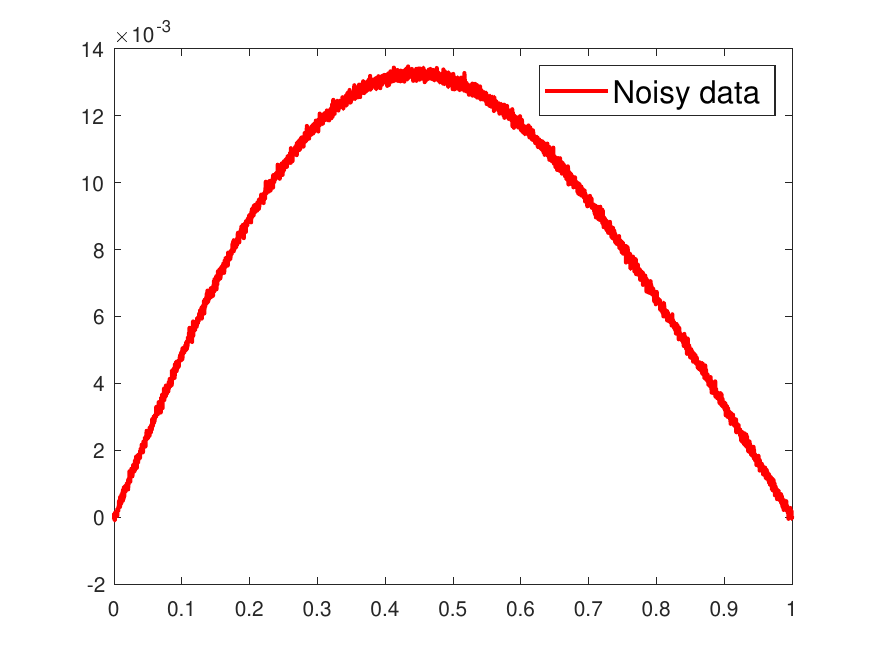}}
	\caption{The true solutions and corresponding noisy observations. The noise level for all four test examples is $\varepsilon=10^{-2}$. From left to right are: (a),(e) Example 1; (b),(f) Example 2; (c),(g) Example 3; (d),(h) Example 4.}
	\label{fig1}
\end{figure}

We demonstrate the performance of WLSQR for regularizing the four linear ill-posed problems. The standard LSQR method is used as a comparison, where the convergence behaviors of the two methods are shown by plotting the variation of relative reconstruction error $ \|x_k-x_{\text{true}}\|_2/\|x_{\text{true}}\|_2$ with respect to $k$. To further show the effectiveness of WLSQR, we also use $x_{\text{true}}$ to find the optimal Tikhonov regularization parameter $\lambda_{opt}$ for \eqref{tikh1} and the corresponding solution $x_{\lambda_{opt}}$, that is $\lambda_{opt}=\min_{\lambda>0}\|x_{\lambda}-x_{\text{true}}\|_{2}$. We use this optimal solution as a baseline for comparing the two methods.  All the experiments are performed using MATLAB R2023b.

\begin{figure}[htbp!]
	\centering
	\subfloat[]{\label{fig:2a}
  \includegraphics[width=0.33\textwidth]{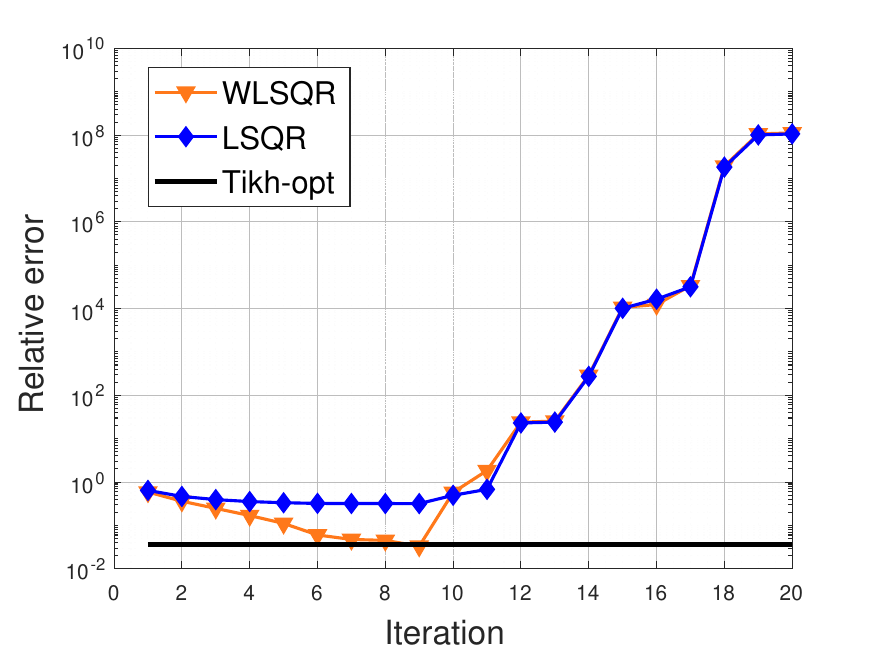}}
	\subfloat[]{\label{fig:2b}
  \includegraphics[width=0.33\textwidth]{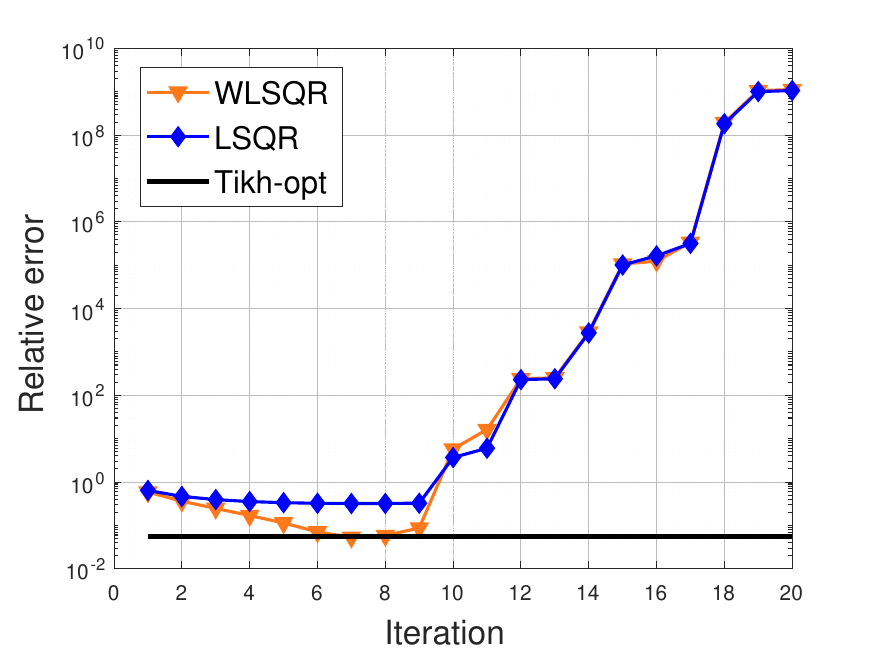}}
	\subfloat[]{\label{fig:2c}
  \includegraphics[width=0.33\textwidth]{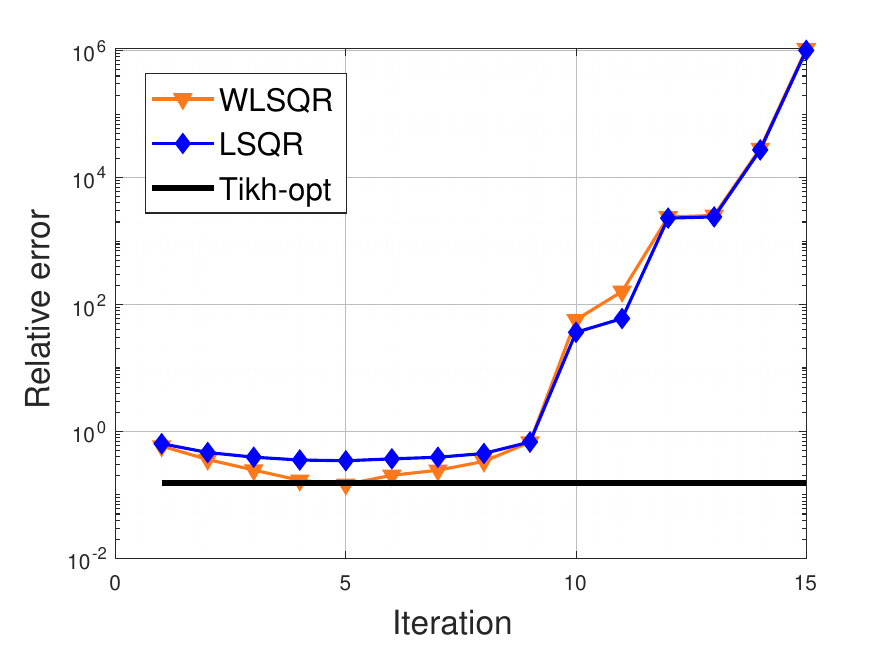}}
  \vspace{-2mm}
	\subfloat[]{\label{fig:2d}
  \includegraphics[width=0.33\textwidth]{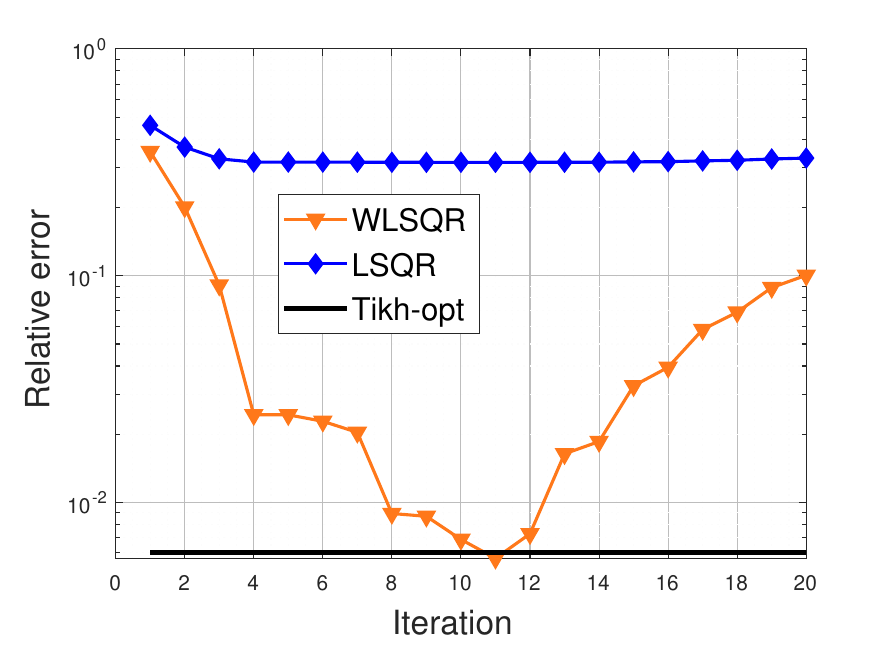}}
	\subfloat[]{\label{fig:2e}
  \includegraphics[width=0.33\textwidth]{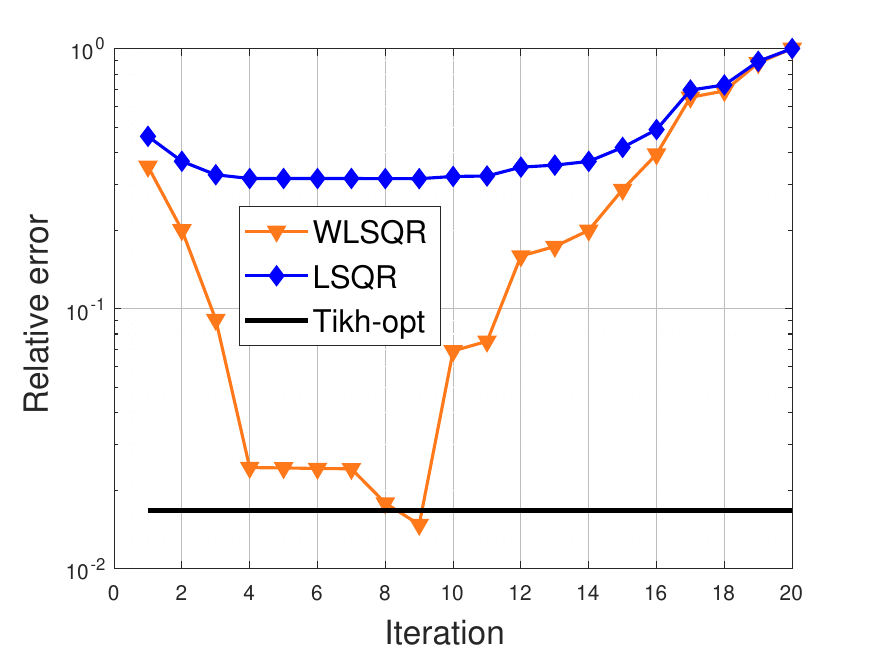}}
	\subfloat[]{\label{fig:2f}
  \includegraphics[width=0.33\textwidth]{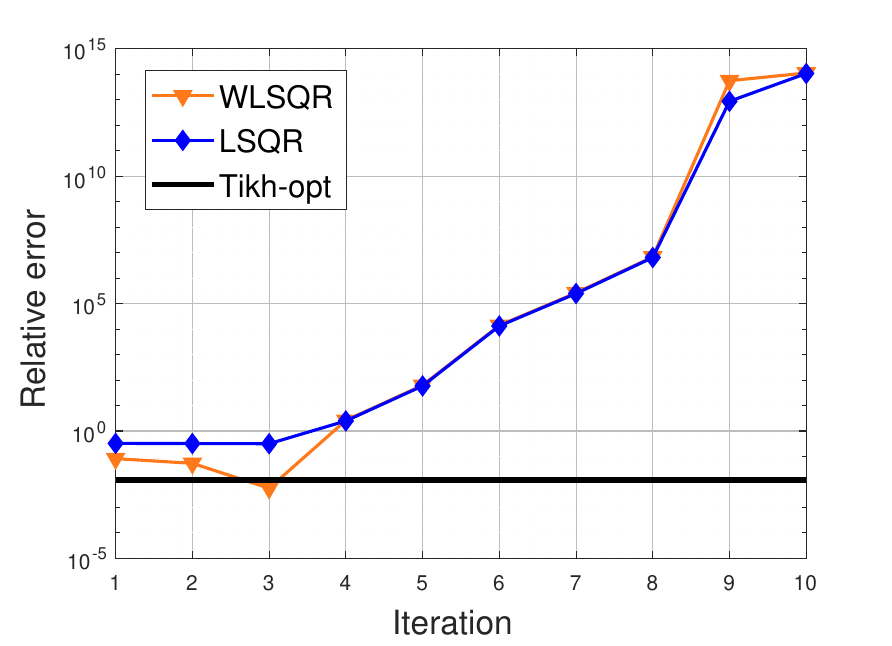}}
  \vspace{-2mm}
  \subfloat[]{\label{fig:2g}
  \includegraphics[width=0.33\textwidth]{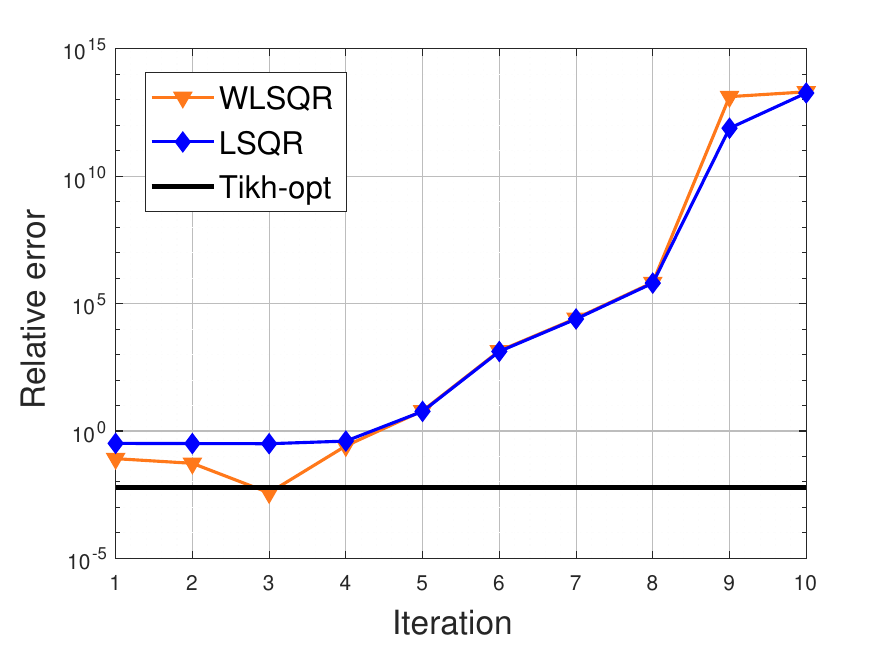}}
	\subfloat[]{\label{fig:2h}
  \includegraphics[width=0.33\textwidth]{./figs/Exam3_semi2.eps}}
	\subfloat[]{\label{fig:2i}
  \includegraphics[width=0.3\textwidth]{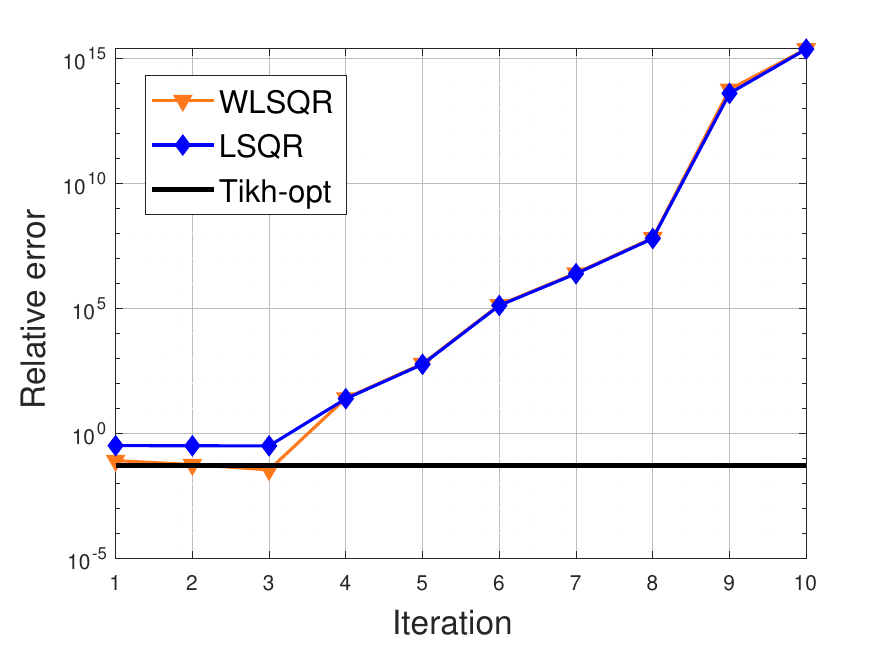}}
  \vspace{-2mm}
  \subfloat[]{\label{fig:2j}
  \includegraphics[width=0.33\textwidth]{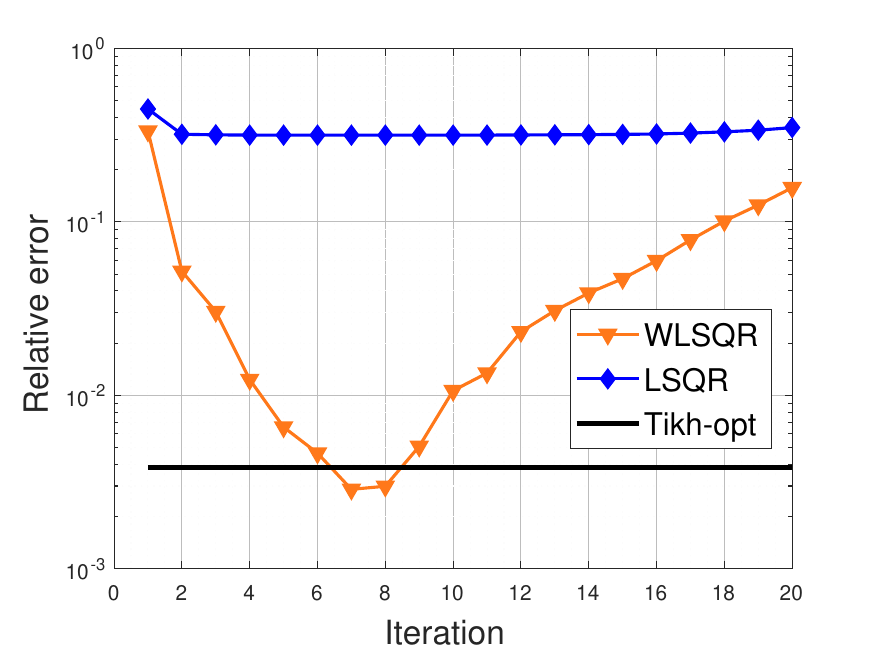}}
	\subfloat[]{\label{fig:2k}
  \includegraphics[width=0.33\textwidth]{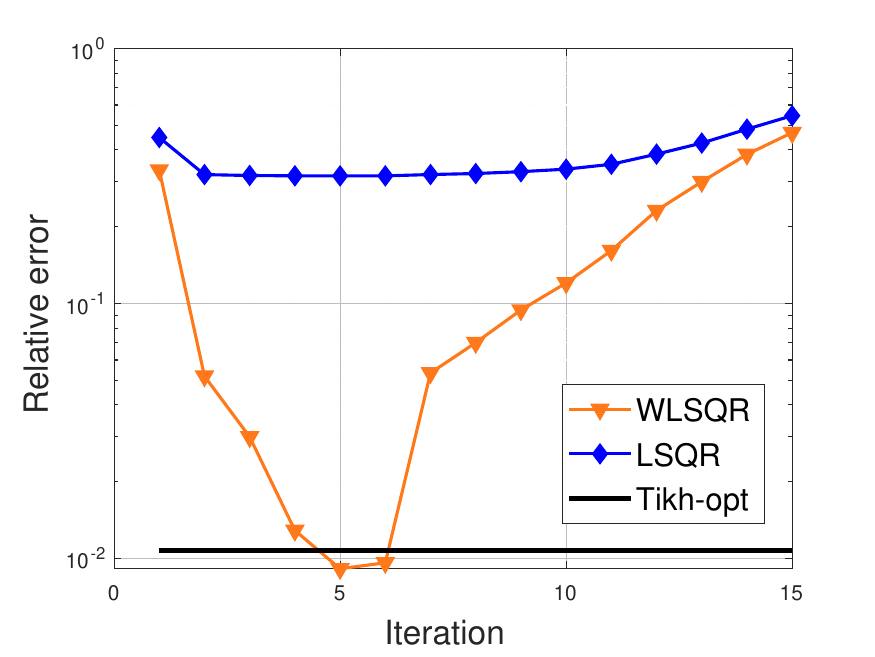}}
	\subfloat[]{\label{fig:2l}
  \includegraphics[width=0.33\textwidth]{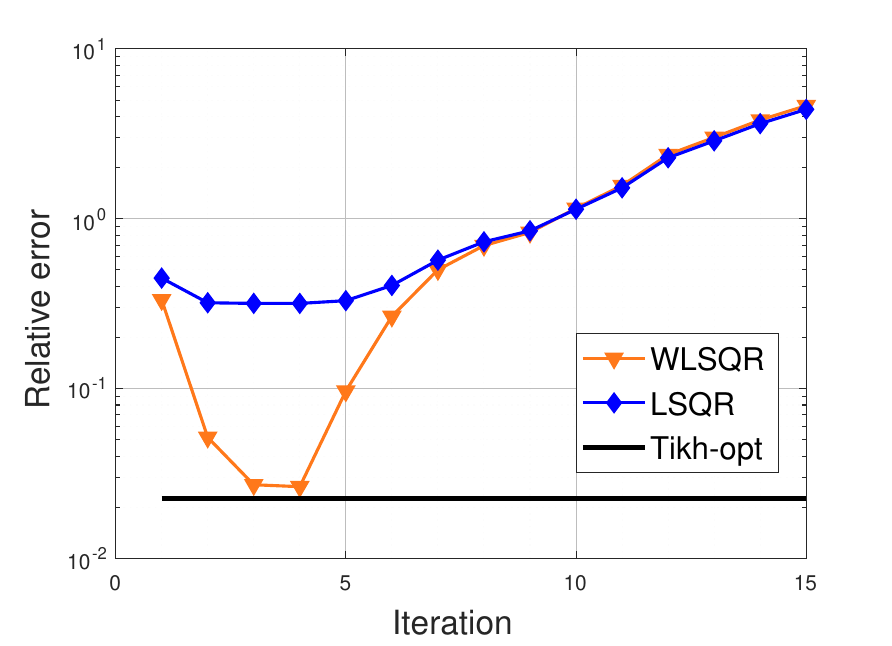}}
	\caption{Semi-convergence curves of LSQR and WLSQR for the four test examples. Figures from the top to bottom correspond to Example 1 -- Example 4; figures from left to right correspond to noise levels $\varepsilon=10^{-3},10^{-2},10^{-1}$.}
	\label{fig2}
\end{figure}

The convergence behaviors of WLSRQ and LSQR are shown in Figure \ref{fig2} using the relative error curves, where the noise levels are set as $\varepsilon=10^{-3},10^{-2},10^{-1}$ for the four test examples. For each example, we find that both the two methods exhibit semi-convergence property, but the relative error of LSQR does not decrease obviously. In contrast, the relative error for WLSQR at the semi-convergence point is much smaller, and it is usually a bit smaller than the best Tikhonov regularization solution. This confirms the regularization effect of WLSQR, which can incorporate the prior information encoded by $\|x\|_{M}^2$ into the solution subspaces.

\begin{table}[htbp]
	\centering
	\caption{Relative errors of the final regularized solutions and corresponding early stopping iterations (in parentheses), where $\varepsilon=10^{-3}$}
	\scalebox{1.0}{
	\begin{tabular}{lllll}
		\toprule
		Problem     & Example 1  & Example 2 & Example 3  & Example 4  \\
		\midrule
		Tikh-opt  & 0.0361 & 0.0060 & 0.0062 & 0.0038  \\
		WLSQR-opt & 0.031 (9) & 0.0057 (11) & 0.0037 (3) & 0.0029 (7)  \\
    WLSQR-DP & 0.0474 (7) & 0.0089 (8) & 0.0538 (2) & 0.0066 (5)   \\
    WLSQR-LC & 0.0451 (8) & 0.0186 (14) & 0.0037 (3) & 0.0233 (12)  \\
    LSQR-opt & 0.3178 (9) & 0.3163 (11) & 0.3166 (3) & 0.3162 (7)   \\
    LSQR-DP  & 0.3194 (7) & 0.3163 (8) & 0.3206 (2) & 0.3163 (5)   \\
    LSQR-LC  & 0.3191 (8) & 0.3164 (14) & 0.3166 (3) & 0.3170 (12)   \\
		\bottomrule
	\end{tabular}}
	\label{tab2}
\end{table}

To see that WLSQR with early stopping can compute a regularized solution of good accuracy, we list the estimated stopping iterations and corresponding relative errors for both WLSQR and LSQR in Table \ref{tab2}. We can find that DP always under-estimate the optimal early stopping iteration for both WLSQR and LSQR, while LC can either under-estimate or over-estimate the optimal early stopping iteration. For WLSQR, the DP method for Example 3 and the LC method for Example 2 and 4 get a regularized solution with a slightly high error, but they are much more accurate than the corresponding LSQR solutions.

\begin{figure}[htbp]
	\centering
	\subfloat[]{\label{fig:3a}
  \includegraphics[width=0.38\textwidth]{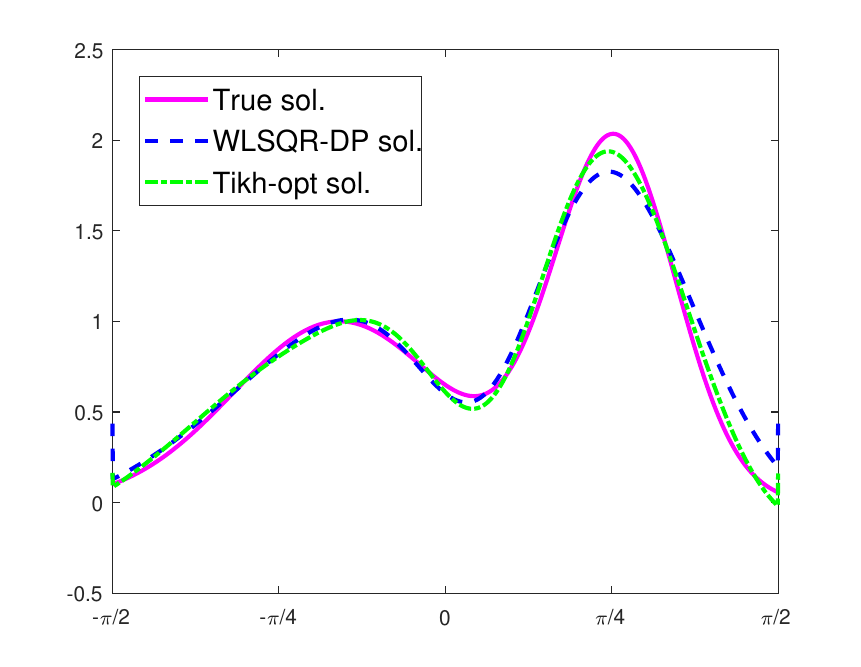}}
	\subfloat[]{\label{fig:3b}
  \includegraphics[width=0.38\textwidth]{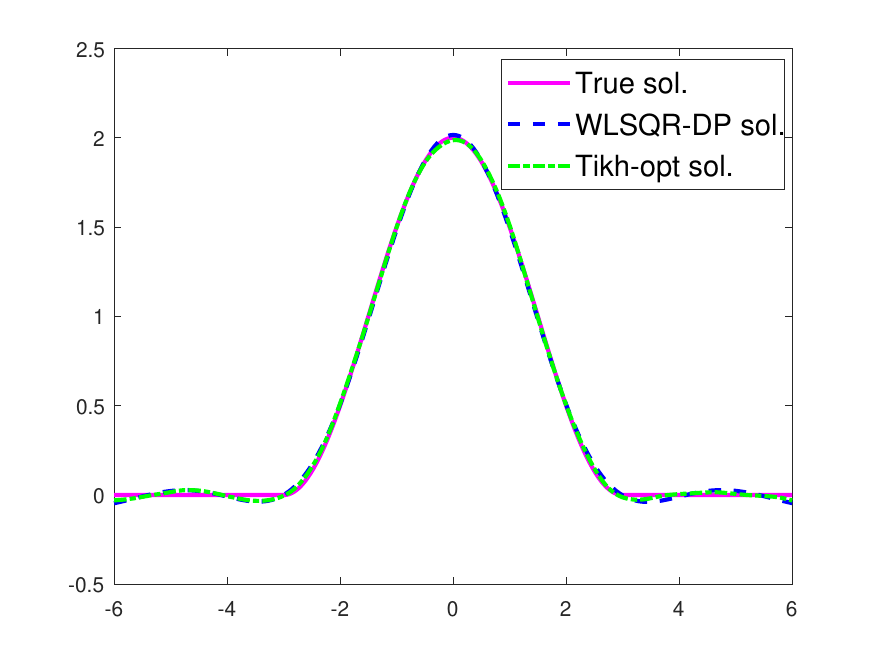}}
  \vspace{-2mm}
	\subfloat[]{\label{fig:3c}
  \includegraphics[width=0.38\textwidth]{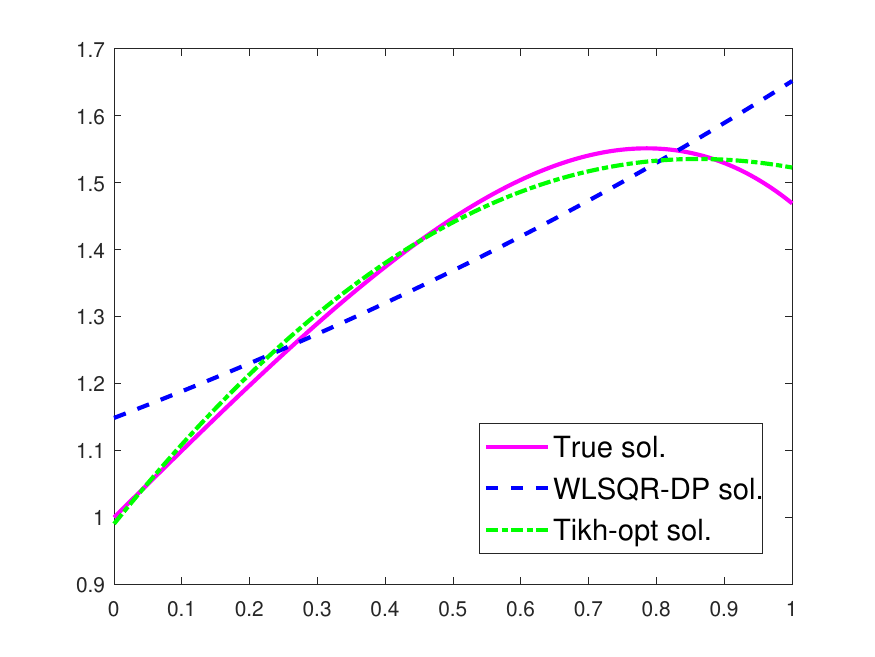}}
	\subfloat[]{\label{fig:3d}
  \includegraphics[width=0.38\textwidth]{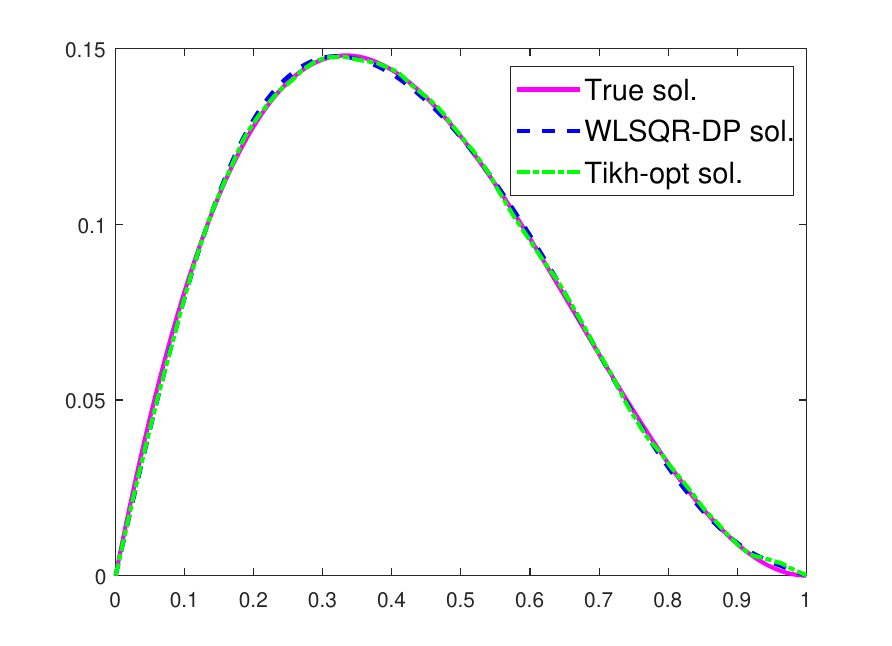}}
	\caption{The reconstructed solutions computed by the WLSQR algorithm with DP as the early stopping rule and the Tikhonov regularization \eqref{tikh1} with optimal regularization parameter. The noise levels are $\varepsilon=10^{-2}$ for all four test examples: (a) Example 1; (b) Example 2; (c) Example 3; (d) Example 4.}
	\label{fig3}
\end{figure}

We depict the reconstructed solutions computed by WLSQR with DP as an early stopping rule in Figure \ref{fig3}, where the optimal Tikhonov regularized solutions are used as a comparison. The reconstructed solutions for Example 1, 2 and 4 by WLSQR are all of high quantity, very close to the optimal Tikhonov solutions and true solutions. For Example 3, the DP solution by WLSQR is over-smoothed, with a slightly poor accuracy. Although we do not depict it, the optimal solution at the semi-convergence point of WLSQR for Example 3 is very close to the true solution. However, in practice, we can not use $x_{\text{true}}$ to find the semi-convergence point, and DP or LC is one of the only several choices we can use to find a solution with not-so-bad accuracy.

\begin{figure}[htbp]
	\centering
	\subfloat[]{\label{fig:4a}
  \includegraphics[width=0.25\textwidth]{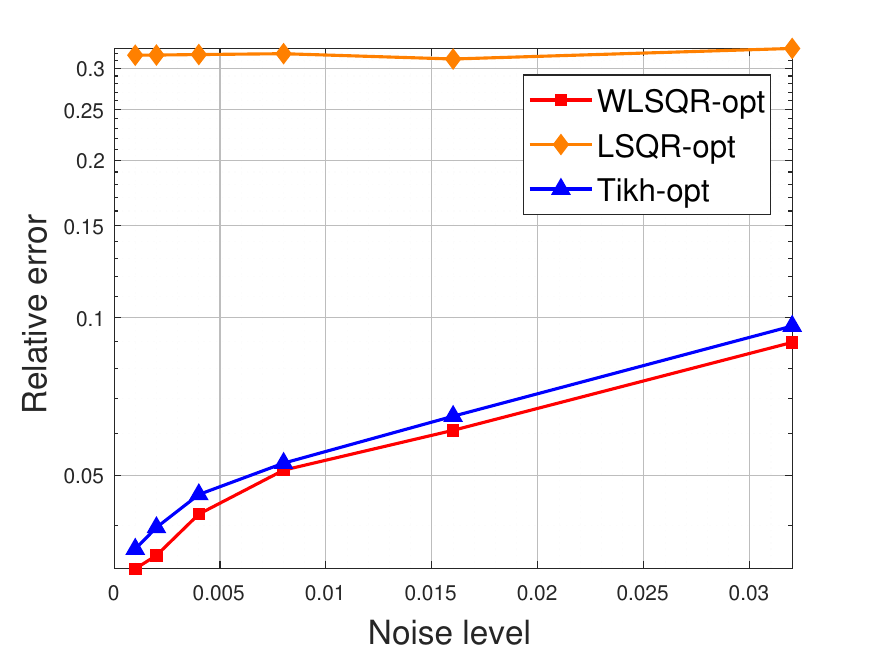}}
	\subfloat[]{\label{fig:4b}
  \includegraphics[width=0.25\textwidth]{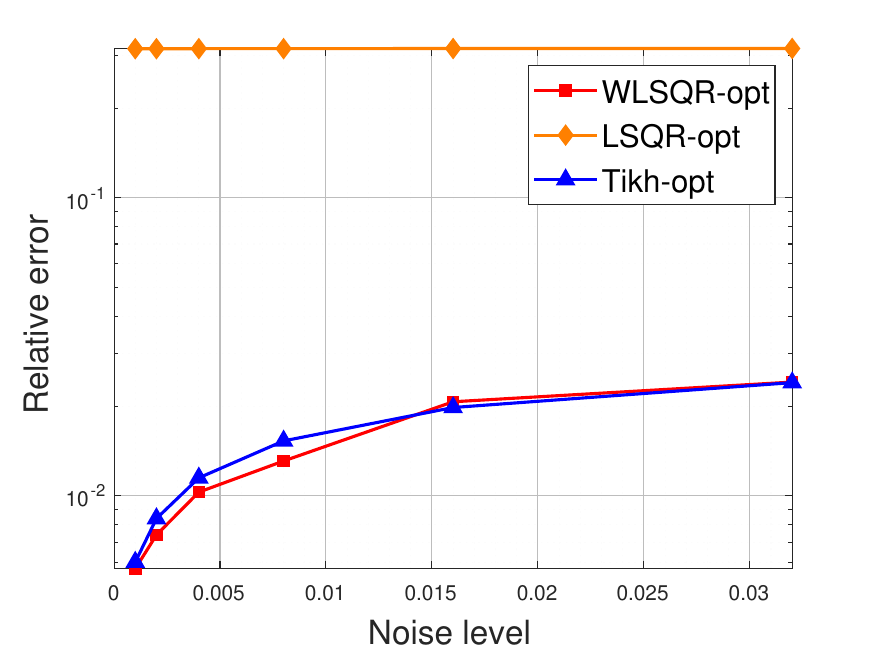}}
	\subfloat[]{\label{fig:4c}
  \includegraphics[width=0.25\textwidth]{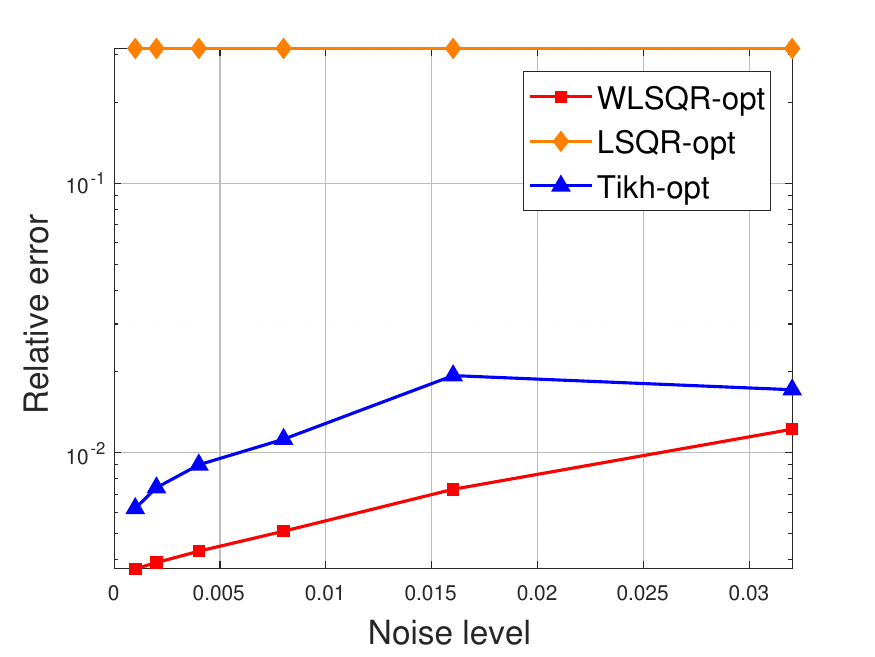}}
	\subfloat[]{\label{fig:4d}
  \includegraphics[width=0.25\textwidth]{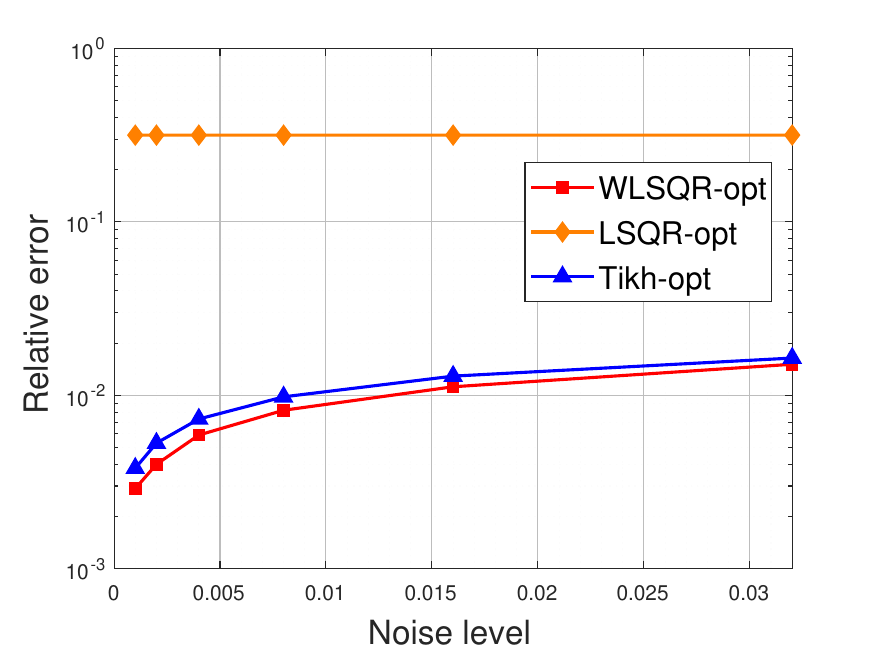}}
  \vspace{-2mm}
  \subfloat[]{\label{fig:5a}
  \includegraphics[width=0.25\textwidth]{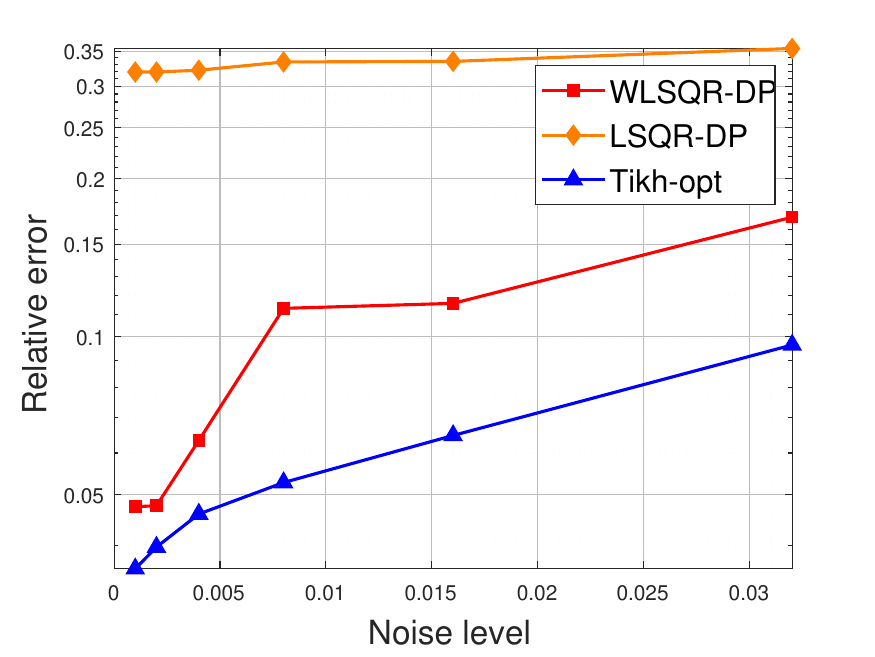}}
	\subfloat[]{\label{fig:5b}
  \includegraphics[width=0.25\textwidth]{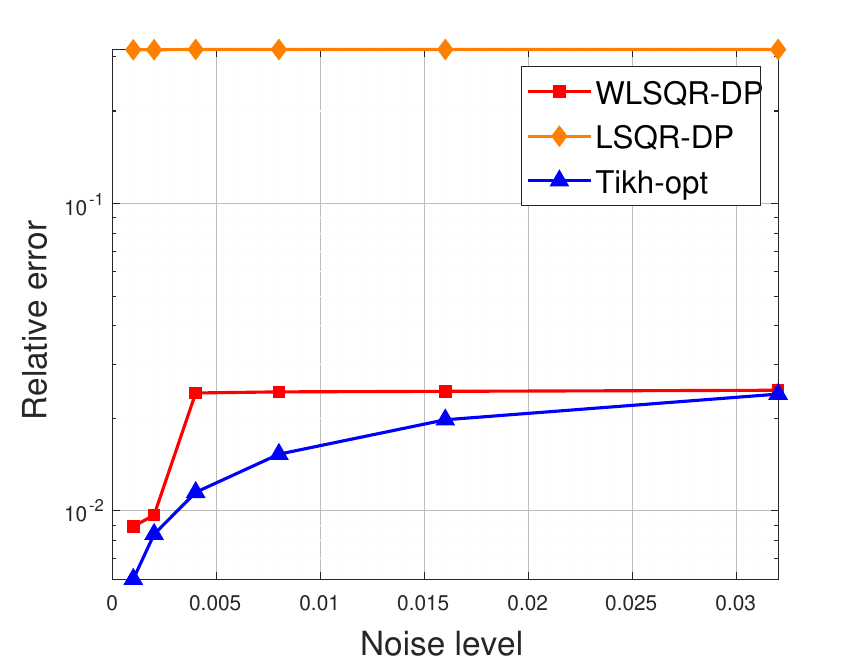}}
	\subfloat[]{\label{fig:5c}
  \includegraphics[width=0.25\textwidth]{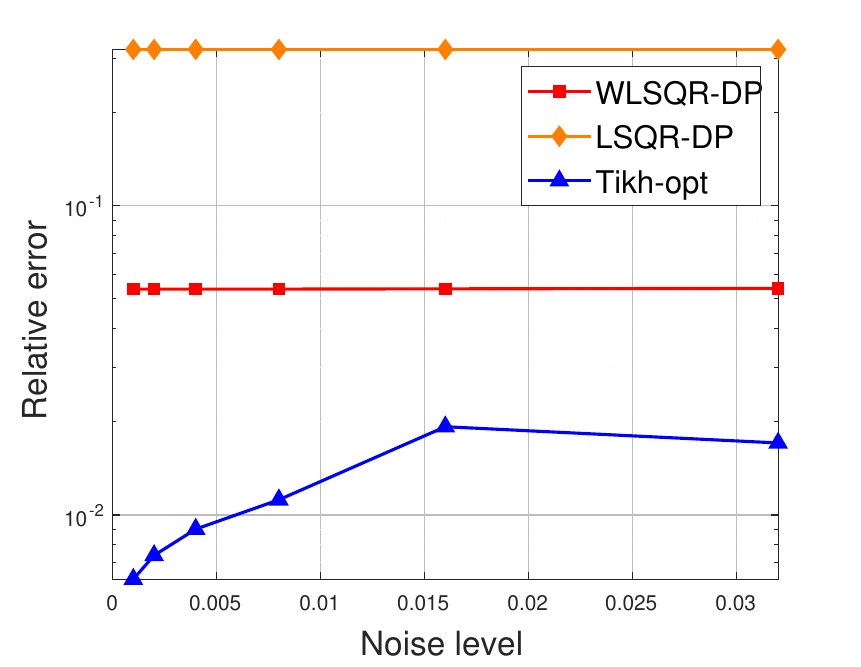}}
	\subfloat[]{\label{fig:5d}
  \includegraphics[width=0.25\textwidth]{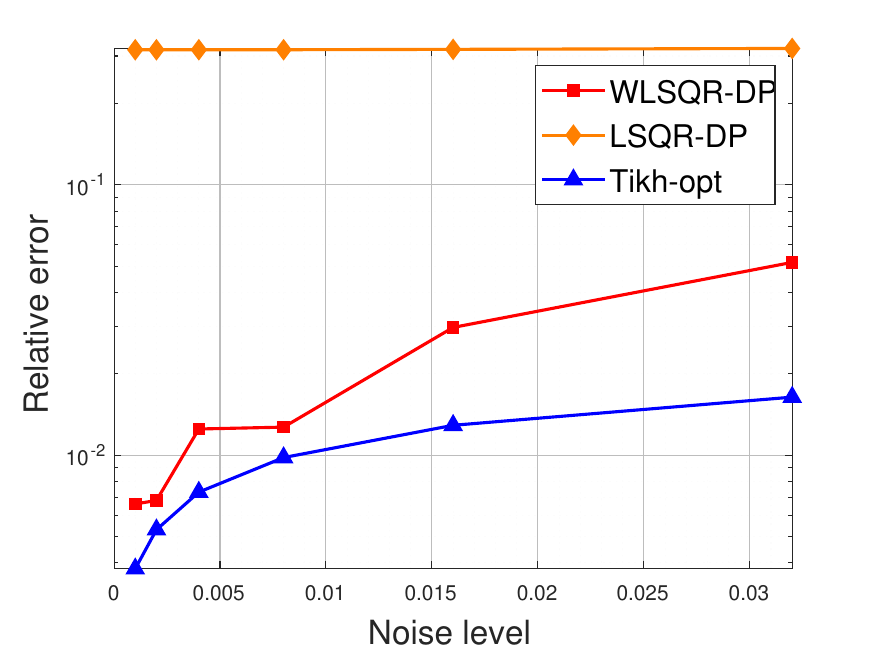}}
	\caption{Convergence rate of the regularized methods WLSQR, LSQR and Tikhonov regularization as the noise level decreases as $\varepsilon=3.2\times 10^{-2}, 1.6\times 10^{-2}, 8\times 10^{-3}, 4\times 10^{-3}, 2\times 10^{-3}, 1\times 10^{-3}$. Figures on the top and bottom correspond to the relative errors at the early stopping iterations that are optimal and estimated by DP, respectively. Figures from left to right correspond to Example 1 -- Example 4.}
	\label{fig4}
\end{figure}

To further demonstrate the regularization effect of WLSQR, we compare the convergence rates of WLSQR with LSQR as the noise level approaches zero, where the noise level decreases as $\varepsilon=3.2\times 10^{-2}, 1.6\times 10^{-2}, 8\times 10^{-3}, 4\times 10^{-3}, 2\times 10^{-3}, 1\times 10^{-3}$. The top four subfigures in Figure \ref{fig4} depict variations of relative error for the optimal regularized solutions computed by WLSQR, LSQR and Tikhonov regularization. We can find that both WLSQR and Tikhonov regularization obtain a convergent regularized solution as the noise decreases to zero, and the WLSQR solution usually has better accuracy. But for LSQR, the relative error remains almost the same even if the noise level approaches zero, this is because LSQR does not make use of the prior information encoded by $\|x\|_{M}^2$. The bottom four subfigures depict variations of relative error when using DP as early stopping rules. Similar to the above tests, DP does not work very well for Example 3, but it is very fruitful for the other three examples.

To summarize, the aforementioned experiments have confirmed that the WLSQR algorithm combined with a proper early stopping rule can obtain a good regularized solution for the regularizer $\|x\|_{M}^2$. This is because WGKB can approximate dominant WSVD components of $A$, which contains main information about the true solution, as is revealed by the Tikhonov solution to \eqref{tikh1} and the TWSVD solution.

\section{Conclusions}\label{sec7}
We have generalized a new form of SVD under a non-standard inner product, named the weighted SVD (WSVD). The WSVD shares several similar properties and applications as the standard SVD, such as the low-rank approximation property and solving the least squares problems. Meanwhile, it is very convenient to handle the matrix computation problems with $\|x\|_{M}$ norm. We have proposed a weighted Golub-Kahan bidiagonalization (WGKB) for computing several dominant WSVD components, and a WGKB-based algorithm, called the weighted LSQR (WLSQR) to solve iteratively least squares problems with minimum $\|x\|_{M}$ norm. Using WSVD, we have analyzed the Tikhonov regularization of the linear ill-posed problem with regularizer $\|x\|_{M}^2$ and given the truncated WSVD solution. We have proposed the WGKB-based subspace projection regularization method, which is equivalent to WLQR with early stopping rules to efficiently compute the regularized solution. This regularization method avoids computing the Cholesky factorization of $M$ and can efficiently incorporate the prior information encoded by the regularizer $\|x\|_{M}^2$. Several numerical experiments are performed to illustrate the fruitfulness of our methods.



\bibliographystyle{abbrv}
\bibliography{ref}

\end{document}